\documentclass[12pt]{amsart}
\usepackage{amssymb,amsmath,amsthm,bm}
\usepackage[colorinlistoftodos,prependcaption,textsize=tiny]{todonotes}
 \definecolor{darkblue}{RGB}{0,0,160}
\usepackage[lining]{libertine}   
\usepackage{cabin}
\usepackage[libertine]{newtxmath}
\usepackage[colorlinks=true,allcolors=darkblue]{hyperref}

\usepackage[T1]{fontenc}

\usepackage{color}

\def\e{{\rm e}}
\def\cic{\mathbf}
\def\eps{\varepsilon}

\def\d{{\rm d}}
\def\dist{{\rm dist}}

\def\R {\mathbb{R}}

\def \l {\langle}
\def \r {\rangle}

\def \and{\qquad\text{and}\qquad}

\newcommand{\supp}{\mathrm{supp}\,}

\newcounter{thms}

\newcounter{other}
\numberwithin{other}{section}
\newtheorem{proposition}[other]{Proposition}
\newtheorem{theorem}[thms]{Theorem}
\newtheorem*{theorem*}{Theorem}
\newtheorem*{proposition*}{Proposition}
\newtheorem{corollary}{Corollary}
\numberwithin{corollary}{thms}
\newtheorem{lemma}[other]{Lemma}
\theoremstyle{definition}
\newtheorem{remark}[other]{Remark}

\def\vint_#1{\mathchoice%
      {\mathop{\kern 0.2em\vrule width 0.6em height 0.69678ex depth -0.58065ex
              \kern -0.8em \intop}\nolimits_{\kern -0.4em#1}}%
      {\mathop{\kern 0.1em\vrule width 0.5em height 0.69678ex depth -0.60387ex
              \kern -0.6em \intop}\nolimits_{#1}}%
      {\mathop{\kern 0.1em\vrule width 0.5em height 0.69678ex depth -0.60387ex
              \kern -0.6em \intop}\nolimits_{#1}}%
      {\mathop{\kern 0.1em\vrule width 0.5em height 0.69678ex depth -0.60387ex
              \kern -0.6em \intop}\nolimits_{#1}}}
\def\vintslides_#1{\mathchoice%
      {\mathop{\kern 0.1em\vrule width 0.5em height 0.697ex depth -0.581ex
              \kern -0.6em \intop}\nolimits_{\kern -0.4em#1}}%
      {\mathop{\kern 0.1em\vrule width 0.3em height 0.697ex depth -0.604ex
              \kern -0.4em \intop}\nolimits_{#1}}%
      {\mathop{\kern 0.1em\vrule width 0.3em height 0.697ex depth -0.604ex
              \kern -0.4em \intop}\nolimits_{#1}}%
      {\mathop{\kern 0.1em\vrule width 0.3em height 0.697ex depth -0.604ex
              \kern -0.4em \intop}\nolimits_{#1}}}

\newcommand{\aveint}[2]{\mathchoice%
      {\mathop{\kern 0.2em\vrule width 0.6em height 0.69678ex depth -0.58065ex
              \kern -0.8em \intop}\nolimits_{\kern -0.45em#1}^{#2}}%
      {\mathop{\kern 0.1em\vrule width 0.5em height 0.69678ex depth -0.60387ex
              \kern -0.6em \intop}\nolimits_{#1}^{#2}}%
      {\mathop{\kern 0.1em\vrule width 0.5em height 0.69678ex depth -0.60387ex
              \kern -0.6em \intop}\nolimits_{#1}^{#2}}%
      {\mathop{\kern 0.1em\vrule width 0.5em height 0.69678ex depth -0.60387ex
              \kern -0.6em \intop}\nolimits_{#1}^{#2}}}

\renewcommand*{\cdots}{%
  \mathinner{{\cdotp}{\cdotp}{\cdotp}}%
}

\numberwithin{equation}{section}

\title[Sparse for rough singular integrals]{A sparse domination principle for rough singular integrals}

\author[J.M.\ Conde-Alonso]{Jos\'e M. Conde-Alonso}
\address{\noindent Departament de Matem\`atiques, Facultat de Ci\`encies, \newline \indent Universitat Aut\`onoma de Barcelona, 08193 Barcelona, Spain}
\email{jconde@mat.uab.cat}
 \author[A.\ Culiuc]{Amalia Culiuc}
 \address{\noindent School of Mathematics, Georgia Institute of Technology, \newline \indent Atlanta, GA 30332, USA}
\email{amalia@math.gatech.edu}
 \author[F. Di Plinio]{Francesco Di Plinio} \address{\noindent Department of Mathematics, University of Virginia,  \newline \indent Kerchof Hall,  Box 400137, Charlottesville, VA 22904-4137, USA   }
 \email{francesco.diplinio@virginia.edu}
 \author[Y.\ Ou]{Yumeng Ou}
 \address{\noindent Department of Mathematics, Massachusetts Institute of Technology, \newline \indent  77 Massachusetts Avenue, Cambridge, MA 02139, USA  }
\email{yumengou@mit.edu}

\thanks{JM Conde-Alonso  was supported in part by ERC Grant 32501 and by MTM-2013-44304-P project. F Di Plinio was partially
supported by the National Science Foundation under the grants
   NSF-DMS-1500449 and  NSF-DMS-1650810.}

\begin{document}
 \begin{abstract} 
We prove that bilinear forms associated to the  rough homogeneous singular integrals  
\[
T_\Omega f(x) = \mathrm{p.v.} \int_{\R^d} f(x-y)   \Omega\left( {\textstyle \frac{y}{|y|}}\right)  \, \frac{\d y}{|y|^d} 
\]
where  $\Omega \in L^q (S^{d-1})$ has vanishing average and $1<q\leq \infty$,  and to Bochner-Riesz means at the critical index in $\R^d$
are dominated by sparse   forms involving $(1,p)$ averages. This domination is stronger  than the weak-$L^1$ estimates   for $T_\Omega$ and for  Bochner-Riesz means, respectively due to Seeger and Christ. Furthermore,  our domination theorems entail as a corollary new sharp quantitative $A_p$-weighted estimates for Bochner-Riesz means and for  homogeneous singular integrals  with unbounded angular part,   extending  previous results of Hyt\"onen-Roncal-Tapiola for $T_\Omega$.
Our results follow from a new abstract sparse domination principle which does not rely on weak endpoint estimates for maximal truncations.
   \end{abstract}
\maketitle

\section{Introduction and main results}

Singular integral operators  of  Calder\'on-Zygmund type, which are a priori \emph{signed} and \emph{non-local}, can be dominated  in norm \cite{Ler2013}, pointwise \cite{CR,Lac2015,LerNaz2015}, or dually \cite{BFP,CuDPOu,CuDPOu2}  
  by 
  sparse averaging operators (forms), which are in contrast \emph{positive} and \emph{localized}.  
For $1\leq p_1,p_2 < \infty$, we call  \emph{sparse $(p_1,p_2)$-averaging  form}  the bisublinear form
\[
\mathsf{PSF}_{\mathcal S;p_1,p_2}(f_1,f_2) := \sum_{Q\in \mathcal S} |Q| \l f_1 \r_{p_1,Q}   \l f_2 \r_{p_2,Q}, \qquad \l f \r_{p,Q}:=  |Q|^{-\frac1p}\left\|f\cic{1}_{Q}\right\|_p,
\]
associated to a (countable) sparse collection $\mathcal S$ of cubes of $\R^d$. The collection $\mathcal S$ is  $\eta$-sparse if  there exist $0<\eta\leq 1$ (a number which will not play a relevant role) and measurable sets $\{E_I:I \in \mathcal S\}$ such that
\[
E_I \subset I, \, |E_I| \geq \eta |I|, \qquad I,J\in \mathcal S, I \neq J \implies E_I \cap E_J  = \varnothing.
\] 

In this article, we prove a sparse domination principle of type
\begin{equation}
\label{sparsegen}
|\l T f_1, f_2 \r| \lesssim \sup_{\mathcal S}\mathsf{PSF}_{\mathcal S;p_1,p_2}(f_1, f_2)
\end{equation} 
for singular integral operators $T$ whose (possible) lack of kernel smoothness forbids the avenue exploited in \cite{Lac2015,Ler2015}. Our principle, summarized in Theorem \ref{theoremABS} below, can be  employed in a rather direct fashion to recover the best known, and sharp, sparse domination results for Dini and H\"ormander type Calder\'on-Zygmund operators \cite{BCDHL,HRT,Lac2015,Li1}.

However, the main purpose of our work is to suitably extend \eqref{sparsegen} to the class  of  rough  singular integrals  introduced in   the seminal paper of Calder\'on and Zygmund \cite{CZ}, and further studied, notably, in Duoandi\-koet\-xea-Rubio de Francia \cite{DR}, Christ \cite{Ch88}, Christ-Rubio de Francia \cite{CRub} and Seeger \cite{Seeger}. Prime examples from this class include the rough homogeneous singular integrals on $\R^d$
\begin{equation}
\label{tomega}
T_\Omega f(x) = \mathrm{p.v.} \int_{\R^d} f(x-y)   \Omega\left( {\textstyle \frac{y}{|y|}}\right)  \, \frac{\d y}{|y|^d},
\end{equation}
with  $\Omega \in L^q (S^{d-1})$  having zero average, as well as the critical   Bochner-Riesz means in dimension $d$, defined by the multiplier operator
\begin{equation}
\label{BRdef}
B_\delta f  = \mathcal F^{-1} \left[ \widehat f(\cdot)\left( 1-|\cdot|^2 \right)^\delta_+ \right], \qquad 
\delta= \frac{d-1}{2}.
\end{equation}
For the singular integrals \eqref{tomega}  no  sparse domination results were known  prior to this article, although some quantitative weighted estimates  were established in the recent works \cite{HRT,PPR}; see below for details.   For the  Bochner-Riesz means \eqref{BRdef}, the recent   results of \cite{BBP} and \cite{CDS} are far from being optimal  at the critical exponent.

 The main difficulty encountered by previous approaches in this setting is the following: first, notice that an estimate of the type \eqref{sparsegen} is already stronger than the weak-$L^{p_1}$  bound for $T$. In particular, if $p_1=1$ then \eqref{sparsegen} recovers the weak-$L^1$ endpoint bound. On the other hand, the preexisting techniques  for  sparse domination  \cite{BBP,BFP,HRT,Lac2015,Ler2015}    essentially rely on  weak-$L^p$ estimates for a   grand maximal truncation  of the singular integral operator $T$. But those do not seem attainable   in the context, for instance, of \cite{Seeger}, as observed in \cite{Ler2015}. In fact, the rough singular integrals we consider below are not known to satisfy such estimate for $p=1$,  and therefore a different approach is required in order to obtain the sparse bounds that we want.

As a corollary of our domination results, we obtain   quantitative $A_p$-weighted estimates for homogeneous singular integrals \eqref{tomega} whose angular part belongs to  $L^q(S^{d-1})$ for some $1<q\leq \infty$. These are novel, and sharp, when $q<\infty$, while in the case $q=\infty$ we recover the best known result recently proved in \cite{HRT} by other methods. Although our result  for the   Bochner-Riesz means \eqref{BRdef} seemingly yields the best known quantitative $A_p$ estimates, we do not know whether our results are  sharp in this case.

\subsection{Main results}  
 Our  main results consist of  estimates for the bilinear forms associated to $T_\Omega$ and $B_\delta$   by sparse operators involving $L^p$-averages.
The formulation of our first theorem requires the Orlicz-Lorentz norms
\[
\|\Omega\|_{L^{q,1}\log L(S^{d-1})} := q \int_0^\infty t\log(\e+t) |\{\theta\in S^{d-1}: |\Omega(\theta)|>t \}|^{\frac 1q} \frac{\d t}{t}    \qquad  1\leq q<\infty.
\] 
\begin{theorem} There exists an absolute dimensional constant $C>0$ such that the following holds. \label{theoremRH}
Let $\Omega\in L^1(S^{d-1})$ have zero average.  Then  for all  $1<t<\infty$,  $f_1\in L^t(\R^d), f_2\in L^{t'}(\R^d)$ there holds \[
 |\l T_\Omega f_1, f_2 \r| \leq \frac{Cp}{p-1} \sup_{\mathcal S} \mathsf{PSF}_{\mathcal S;1,p}(f_1, f_2) \begin{cases}  \|\Omega\|_{L^{q,1}\log L(S^{d-1})}  & 1< q<\infty, \quad p\geq q', \\   \|\Omega\|_{L^\infty(S^{d-1})}  &   1<p<\infty.   \end{cases}  \]
\end{theorem}
\begin{remark} To avoid Lorentz norms in the statement, one may recall the continuous embeddings $L^{q+\eps}(S^{d-1})\hookrightarrow L^{q,1}\log L(S^{d-1})\hookrightarrow L^{q} (S^{d-1})  $ for all $1\leq q<\infty$ and $\eps>0$.
\end{remark}
\begin{theorem} \label{theoremBR} There exists an absolute dimensional constant $C>0$ such that the following holds. For all $1<t<\infty$,  $f_1\in L^t(\R^d), f_2\in L^{t'}(\R^d)$,
the critical Bochner-Riesz means \eqref{BRdef} satisfy\[
 |\l  B_\delta f_1, f_2 \r| \leq \frac{Cp}{p-1} \sup_{\mathcal S} \mathsf{PSF}_{\mathcal S;1,p}(f_1, f_2), \qquad 1<p<\infty.   \]
\end{theorem}
The weak-$L^1$ estimate for $T_\Omega$ is the main result of \cite{Seeger}, while the same endpoint estimate for \eqref{BRdef} has been established in \cite{Ch88}. Theorems \ref{theoremRH} and \ref{theoremBR} recover such results; see Appendix \ref{Weak11} for a proof of this implication,  which we include for future reference. This is not surprising as the localized estimates for \eqref{tomega}, \eqref{BRdef} which are needed to apply our abstract result are a distillation and an improvement of the microlocal techniques of \cite{Seeger} and of the previous works \cite{Ch88,CRub}, and of the oscillatory integral estimates of \cite{Ch88} respectively.

We reiterate that the commonly used techniques for sparse domination, which rely on the weak-$L^1$ estimate for the
maximal truncation of the singular integral operator, fail to be applicable in the context of  Theorem \ref{theoremRH} as the maximal truncations of $T_\Omega$ in \eqref{tomega} are not known to satisfy such estimate even when $\Omega \in L^\infty(S^{d-1})$ \cite{GS2}. Our abstract Theorem \ref{theoremABS}, whose statement is more technical and is postponed until Section \ref{secABS}, only relies on the uniform $L^2$ (or $L^r$ for any $r$) boundedness of the truncated operators, and thus might be considered stronger than the approaches of the mentioned references. See Remark \ref{Remphy} for additional discussion on this point. 

Theorems \ref{theoremRH} and \ref{theoremBR} entail as corollaries  a family of  quantitative weighted estimates.  
\setcounter{thms}{1}
\begin{corollary} \label{corRH1}  If $\Omega$ lies in the unit ball of $L^{q,1}\log L(S^{d-1})$ for some $1<q<\infty$ and has zero average, we have the weighted norm inequalities
\begin{equation} \label{weighted}
 \|T_\Omega\|_{L^t(w) \to L^t(w)} \leq C_ {t,q}   [w]_{A_{\frac{t}{q'}}} ^{\max\left\{1,\frac{1}{t-q'}\right\}}, \qquad   \;q'<t<\infty.
\end{equation}
If furthermore  $\|\Omega\|_{L^\infty(S^{d-1})}\leq 1$,  
\begin{equation} \label{weightedinfty}
 \|T_\Omega\|_{L^t(w)\to L^t(w)} \leq C_t
[w]_{A_{t}}^{\frac{1}{t-1}\max\left\{t,2\right\}} \qquad   1<t<\infty.  
\end{equation}
\end{corollary}
\setcounter{corollary}{0}
\setcounter{thms}{2}
\begin{corollary} \label{corBR1}  Referring to \eqref{BRdef}, we have the weighted norm inequalities
\begin{equation}  \label{weightedBR}
 \|B_\delta\|_{L^t(w)\to L^t(w)} \leq C_t
[w]_{A_{t}}^{\frac{1}{t-1}\max\left\{t,2\right\}} \qquad   1<t<\infty.  
\end{equation}
\end{corollary}
\begin{proof}[Proof of Corollaries \ref{corRH1}, \ref{corBR1}]
To prove (\ref{weighted}), applying Theorem \ref{theoremRH} for $p=q'$ (strictly speaking, to the adjoint of $T_\Omega$) yields that the bilinear form associated to $T_\Omega $ is dominated by  \[\sup_{\mathcal S}\mathsf{PSF}_{\mathcal S;q',1}.\] The proof of the weighted estimate can then be found, for instance, in \cite[Proposition 6.4]{BFP}.
We prove (\ref{weightedinfty}), and  \eqref{weightedBR} follows via the same argument: below, $C$ denotes a  positive absolute constant which may vary between occurrences.   Combining the inequality \cite[Proposition 4.1]{DpLer2013}
\[
\l f \r_{1+\eps,Q} \leq \l f \r_{1,Q} +C\eps \l \mathrm{M}_{1+\eps}f \r_{1,Q},
\]
which is valid for all $\eps>0$, with the estimate of Theorem \ref{theoremRH} for $p=1+\eps$  we obtain
 \[
 |\l T_\Omega f_1, f_2 \r| \leq \frac{C}{\eps}   \sup_{\mathcal S} \mathsf{PSF}_{\mathcal S;1,1}(f_1, f_2)   + C \sup_{\mathcal S} \mathsf{PSF}_{\mathcal S;1,1}(\mathrm{M}_{1+\eps}f_1, f_2), \qquad \eps>0. 
 \]
The above display leads via standard reasoning  \cite{CMP,HPR,KM} to the chain of inequalities
\[\begin{split}
\|T\|_{L^t(w) \to L^t(w)} &\leq C_t [w]_{A_t}^{\max\left\{1,\frac{1}{t-1}\right\}} \inf_{0<\eps<t-1} \left(\frac{1}{\eps} + \|\mathrm{M}_{1+\eps}\|_{L^t(w)\to L^t(w)} \right) \\ &\leq 
C_t [w]_{A_t}^{\max\left\{1,\frac{1}{t-1}\right\}}\inf_{0<\eps<t-1} \left(\frac{1}{\eps} + [w]_{A_{\frac{t}{1+\eps}}}^{\frac{1+\eps}{t-(1+\eps)}} \right)   \leq 
  C_t
[w]_{A_{t}}^{\frac{1}{t-1}\max\left\{t,2\right\}},
\end{split}
\] 
and the proof is complete.
\end{proof}
 Our Corollary \ref{corRH1} is a quantification of the weighted inequalities
due to Watson \cite{Wat1990} and Duoandikoetxea \cite{Duo1993}: 
if $1<q\leq \infty$ and $\Omega \in L^q(S^{d-1})$ then  
\[
 \left.\begin{array}{ll} w\in A_{\frac{t}{q'}} & q'\leq t< \infty,\, t\neq 1, \\  w^{\frac{1}{1-t}}\in A_{\frac{t'}{q'}} & 1< t\leq q, \, t\neq \infty, \\  w^{q'} \in A_t & 1<t<\infty. \end{array} \right\} \implies \|T_\Omega\|_{L^t(w) \to L^t(w)}< \infty
\]
Estimate \eqref{weightedinfty} was first established by Hyt\"onen, Roncal and Tapiola \cite{HRT} via a different two-step technique involving sparse domination for Dini-type kernels, a Littlewood-Paley decomposition along the lines of \cite{CRub} and interpolation with change of measure.  In \cite{PPR}, these ideas were extended to obtain $A_1$ estimates for $T_\Omega$ and commutators of $T_\Omega$ and BMO symbols. At this time, we do not know whether the power of the Muckenhoupt constant in \eqref{weightedinfty} is sharp.

Qualitative $A_p$-bounds for critical Bochner-Riesz means are classical \cite{ShiSun}; see also \cite{V}. On the other hand, Corollary \ref{corBR1} seems to be the first quantitative $A_p$ estimate for $B_\delta$. We do not know whether the power of the $A_p$ constant in \eqref{weightedBR} is sharp;  the construction in \cite[Corollary 3.1]{LPR} shows that the optimal power $\alpha_p$  must obey $\alpha_p\geq \max\{1,1/(p-1)\}$. The   article \cite{BBP} contains sparse domination estimates and weighted inequalities for the supercritical regime $0<\delta'<\delta$ which are not informative in the critical case. An extension of our methods to the supercritical cases will appear in   forthcoming work.

Finally, we mention that our argument for \eqref{weightedinfty} and \eqref{weightedBR}   shows  that improvements of such power in Corollaries \ref{corRH1} and \ref{corBR1} are tied to the blowup rate as $p\to 1^+$ of the main estimate of Theorems \ref{theoremRH} and \ref{theoremBR}. 
 

\subsection{A remark on the proof and plan of the article}  Theorems \ref{theoremRH} and \ref{theoremBR} fall under the scope of the same abstract result, Theorem \ref{theoremABS}, which is stated and proved in Section \ref{secABS}. Theorem \ref{theoremABS} is obtained by means of an iterative scheme reminiscent of the arguments used in \cite{CuDPOu} by three of us to prove a sparse domination estimate for the bilinear Hilbert transform, and later adapted to dyadic and continuous Calder\'on-Zygmund singular integrals in \cite{CuDPOu2}. At each iteration, a Calder\'on-Zygmund type decomposition is performed, and the operator itself is decomposed into small scales (scales falling within the exceptional set) which will be estimated at subsequent steps of the iteration, and large scales. The action of the large scales on the good parts is controlled by means of the uniform $L^r$-bound for the truncations of $T$. The contribution of the bad, mean zero part under the large scales of the operator is then controlled by means of suitably localized estimates relying on the \emph{constant-mean zero} type cancellation.  We emphasize that the present work shares a   perspective based on bilinear forms with  other recent papers: \cite{KL} by Krause and Lacey and \cite{LS} by Lacey and Spencer. The notable difference is that these references, dealing  with oscillatory and random discrete singular integrals, use (dilation) symmetry breaking and  $TT^*$, rather than constant-mean zero, as the principal cancellation mechanisms, in accordance with the oscillatory nature of their objects of study.


  Section \ref{secDH}   contains    localized estimates for kernels of Dini and H\"ormander type which, besides being of use in later arguments, allow us to reprove the optimal sparse domination results for these classes: we send to Subsection \ref{ssDH} for the statements.
  In Sections \ref{secRH}  and \ref{secBR} we  provide the  necessary localized estimates for Theorems \ref{theoremRH} and \ref{theoremBR} respectively.  The estimates of Section \ref{secRH} are a delicate strengthening of the microlocal arguments of \cite{Seeger}. 
The proof of Theorem \ref{theoremBR},  a re-elaboration along the same lines of the arguments of \cite{Ch88}, is carried out in Section \ref{secBR}.
  Although we find  hard to believe that  these techniques can be sharpened towards the stronger localized $(1,1)$ estimate, we have no explicit counterexample for this possibility. 
 \subsection*{Notation} 
 As customary, $q'=\frac{q}{q-1}$ denotes the Lebesgue dual exponent to $q\in (1,\infty)$, with the usual extension $1'=\infty$, $ \infty'=1$.
 We denote the center and the sidelength of  a cube $Q\in \R^d$ by  $c_Q$ and  $\ell(Q) $ respectively. We will also adopt the shorthand $s_Q=\log_2 \ell (Q)$. 
We write 
\[
\mathrm{M}_{p}(f)(x) =\sup_{Q \subset \R^d} \l f \r_{p,Q} \cic{1}_Q(x)
\]
for the $p$-Hardy Littlewood maximal function. The positive constants  implied by  the almost inequality sign $\lesssim$ may depend (exponentially) on the dimension $d$ only and    may vary from line to line without explicit mention.
  \subsection*{Acknowledgments} This  work was initiated while J.\ M.\  Conde-Alonso was visiting the Department of Mathematics at the University of Virginia during the Fall semester of 2016; the kind hospitality of the Department is gratefully acknowledged.
  
   The authors have greatly benefited from their participation in the online seminar series on Sparse Domination of Singular Integral Operators hosted by Georgia Tech in the Fall 2016 semester. In particular, they would like to thank Ben Krause, Michael Lacey, and Dario Mena Arias for their inspiring seminar talks on their current and  forthcoming works related to sparse domination in the arithmetic setting.

The authors are grateful to Alexander  Barron and Jill Pipher for stimulating discussions on the subject of this article which led to a simplification of one of the assumptions in Theorem \ref{theoremABS}, and to David Cruz-Uribe, Javier Duoandikoetxea and Carlos P\'erez for providing   useful references and remarks concerning the weighted theory of rough singular integrals. 
  
\section{A sparse domination principle} \label{secABS}

This section is dedicated to the statement and proof of our sparse domination principle, Theorem \ref{theoremABS}. 
 \subsection{The main structural assumptions}
 Our structural assumptions in Theorem \ref{theoremABS} will be the following. Let $1< r<\infty$ and $\Lambda$ be an $L^{r}(\R^d)\times L^{r'}(\R^d)$-bounded bilinear form whose kernel $K=K(x,y)$ coincides with a   function away from the diagonal   $\{(x,y)\in \R^d\times\R^d: x=y\}$.  More precisely,  whenever $f_1 \in L^{r}(\R^d), f_2\in L^{r'}(\R^d)$ are compactly and  disjointly supported
 \[
 \Lambda (f_1,f_2)=\int_{\R^d} \int_{\R^d} K(x,y) f_1(y) \,\d y \, f_2(x)\,  \d x
 \]
with absolute convergence of the integral. 
We  assume that there exists $1<q\leq \infty$ such that the kernel $K$ of $\Lambda$ admits the decomposition
\begin{equation}
\label{assker2} \tag{$\mathsf{SS}$}
\begin{split}
&  
K(x,y)=\sum_{s\in \mathbb Z} K_s(x,y), \qquad 
\\ &
\supp K_s \subset\left\{(x,y)\in \R^d\times \R^d: x-y \in A_s\right\}, \qquad A_s:=\left\{z\in \R^d: 2^{s-2}< |z|<2^{s}\right\},
 \\
& [K]_{0,q} :=\sup_{s\in \mathbb Z} 2^{\frac{sd}{q'}}\sup_{x\in \R^d}  \left(\|K_s(x,x+\cdot)\|_{q} +   \|K_s(x+\cdot,x)\|_{q} \right)<\infty. 
\end{split}
\end{equation}
Further, we assume that the   truncated forms  associated to the above decomposition by
\begin{equation}
\label{ltf}\Lambda_{\mu}^\nu(h_1,h_2) := \int\sum_{\mu  < s\leq \nu  }  K_s(x,y)h_1(y)h_2(x) \, \d y\d x 
\qquad \mu,\nu\in \mathbb Z \cup\{-\infty,\infty\}
\end{equation}
satisfy
\begin{equation}
\label{asstrunc} \tag{$\mathsf{T}$}
C_{\mathsf{T}}(r)=:\sup_{\mu<\nu} \left(\|\Lambda_{\mu}^\nu\|_{L^r(\R^d)\times L^{r'}(\R^d)\to\mathbb{C}}\right) <\infty.
\end{equation}
\begin{remark} \label{remlimarg1}
Under the assumptions  \eqref{assker2} and \eqref{asstrunc},   a standard limiting argument \cite[Par.\ I.7.2]{Stein}  yields    that\[
\Lambda (f_1,f_2) = \l m f_1,\overline {f_2}\r     +\lim_{\nu\to \infty} \Lambda_{-\nu}^{\nu}(f_1,f_2) 
\]
for some   $m\in L^\infty(\R^d)$, whenever $f_1 \in L^{r}(\R^d), f_2\in L^{r'}(\R^d)$. It is not hard to see \cite[Lemma 4.7]{LMena} that  \[|\l m f_1,f_2\r|\lesssim\|m\|_\infty \sup_{\mathcal S}\mathsf{PSF}_{\mathcal S;1,1}(f_1, f_2) \]
so that for the purpose of our Theorem \ref{theoremABS} below we may assume that $m=0 $ in the above equality. For this reason,
when $\mu=-\infty$ or $\nu=\infty$ or both, we are allowed to omit the subscript or superscript in \eqref{ltf} and simply write   $\Lambda^\nu$ or $\Lambda_\mu$ or $\Lambda$. Also, when $\mu\geq \nu$, the summation in \eqref{ltf} is void, so that $\Lambda_\mu^\nu\equiv 0$. \end{remark}
\subsection{Localized spaces over stopping collections}
A further condition in our abstract theorem will involve local norms associated to \emph{stopping collections} of (dyadic) cubes. Throughout the article, by \emph{dyadic cubes} we refer to the elements of any fixed  dyadic lattice  $\mathcal D$ in $\R^d$.

  Let $Q\in \mathcal D$ be a fixed dyadic cube in $\R^d$. A   collection $ \mathcal Q\subset \mathcal D$ of dyadic cubes is a \emph{stopping collection} with top $Q$
if the elements of $\mathcal Q$ are pairwise disjoint and contained in $3Q$,
\begin{equation}
\label{stopprop}
L,\,L'\in \mathcal Q,\; L\cap L'\neq \emptyset\implies L=L', \qquad L\in \mathcal Q \implies L \subset 3Q
\end{equation}
 and enjoy the further  separation properties
\begin{equation}
\label{separation}
\begin{split}
&
L, L'\in \mathcal Q,\, |s_L- s_{L'}| \geq  8\implies 7L \cap 7L'=\varnothing,\\
&
 \bigcup_{L\in \mathcal Q: 3L\cap 2Q\neq \emptyset } 9L \subset   \bigcup_{L\in \mathcal Q} L=:\mathsf{sh} \mathcal{Q};\end{split}
\end{equation}
the notation $\mathsf{sh} \mathcal{Q}$ for the union of the cubes in $\mathcal Q$ will also be used below.
For $1\leq p\leq \infty$, define $\mathcal Y_p(\mathcal Q)$ to be the subspace of $L^p(\R^d)$ of functions satisfying
\[
\supp h \subset 3Q, \qquad \infty>
\|h\|_{\mathcal Y_p(\mathcal Q) }:=\begin{cases} \displaystyle \max\left\{\left\|h\cic{1}_{\R^d\setminus \mathsf{sh} \mathcal{Q} }\right\|_{\infty} ,\,\sup_{L\in \mathcal Q}\,\inf_{x\in \widehat L } \mathrm{M}_p h(x) \right\} & p<\infty \\ \|h \|_{\infty} & p=\infty 
\end{cases} \]
where we wrote $\widehat L$ for the (non-dyadic) $2^{5}$-fold dilate of $L$.
We also   denote by ${\mathcal X_p(\mathcal Q) }$ the subspace of $\mathcal Y_p(\mathcal Q)$ of functions satisfying
\[
b=\sum_{L\in \mathcal Q} b_L, \qquad \supp  b_L \subset L.
\]
Furthermore, we write $b\in \dot{\mathcal X}_{p}(\mathcal Q)$ if 
\[
b\in  {\mathcal X}_{p}(\mathcal Q), \qquad \int_{L} b_L = 0\quad \forall L \in \mathcal Q.
\]
We will use the notation $\|b\|_{\mathcal X_p(\mathcal Q)}$ for $\|b\|_{\mathcal Y_p(\mathcal Q)}$ when $b\in \mathcal X_p(\mathcal Q)$, and similarly for $b\in \dot{\mathcal X}_{p}(\mathcal Q)$. When the stopping collection $\mathcal Q$ is clear from the context or during proofs we may omit $(\mathcal Q)$ from the subscript and simply write $\|\cdot\|_{\mathcal Y_p  }$ or $\|\cdot\|_{\mathcal X_p}$.
\begin{remark}[Calder\'on-Zygmund decomposition] \label{CZrem}
 There is a natural Cald\'eron-Zygmund decomposition associated to stopping collections. Observe that if $\mathcal Q$ is a stopping collection   there holds
\[
\sup_{L \in \mathcal Q} \l h \r_{p,L} \leq 2^{5d} \|h\|_{\mathcal Y_p(\mathcal Q)}.
\]
Therefore, we may decompose $h\in \mathcal Y_p(\mathcal Q)$ as
\[
h = g+ b, \qquad   b=\sum_{L\in \mathcal{Q}} b_{ L}, \qquad b_{ L}=\left( h- { \frac{1}{|L|} \int_{L} h(x) \, \d x}\right) \cic{1}_{L}\]
such that
\[
\|g\|_{\mathcal Y_\infty(\mathcal Q)} \leq  2^{5d}\|h\|_{\mathcal Y_p(\mathcal Q)}, \qquad b\in \dot{ \mathcal{X}}_p(\mathcal{ Q}), \; \|b\|_{\dot{ \mathcal{X}}_p(\mathcal{ Q})} \leq 2^{5d+1}\|h\|_{\mathcal Y_p(\mathcal Q)}.
\]
These are nothing else but the usual properties of the Cald\'eron-Zygmund decomposition rewritten in our context.
\end{remark}
 \subsection{The statement} Before stating our result,  we introduce the notation
\begin{equation} \label{therealLambdaq}
\Lambda_{Q,\mu,\nu}(h_1,h_2): =\Lambda^{\min\{s_{Q},\nu\}}_
\mu(h_1\cic{1}_{Q},h_2) = \Lambda^{\min\{s_{Q},\nu\}}_
\mu(h_1\cic{1}_{Q},h_2\cic{1}_{3Q})
\end{equation}
for all dyadic cubes $Q$; the last equality in \eqref{therealLambdaq} is a consequence of the assumptions on the support of  $K_s$ in \eqref{assker2}.
Furthermore, given a stopping collection $\mathcal Q$ with top $Q$, we define the truncated forms
\begin{equation} \label{therealLambdaQ}\begin{split}
\Lambda_{ \mathcal Q,\mu,\nu}(h_1,h_2)& := \Lambda_{Q,\mu,\nu}(h_1,h_2) - \sum_{{\substack{L\in \mathcal Q\\ L\subset Q}}} \Lambda_{L,\mu,\nu}(h_1,h_2)  = \Lambda_{ \mathcal Q,\mu,\nu}(h_1\cic{1}_Q,h_2\cic{1}_{3Q}).\end{split}
\end{equation}
Again, the last equality is due to the support of  $K_s$ in \eqref{assker2}.
A further consequence of assumptions \eqref{assker2}, \eqref{asstrunc} is that the forms $\Lambda_{ \mathcal Q,\mu,\nu}$ satisfy uniform bounds on $\mathcal Y_r(\mathcal Q) \times \mathcal Y_{r'}(\mathcal Q) $.
\begin{lemma} \label{lemma0} There exists a  positive absolute  constant $
\vartheta$ such that   
\[
\left|\Lambda_{ \mathcal Q,\mu,\nu}(h_1,h_2) \right|\leq 2^{\vartheta d}C_{\mathsf{T}}(r)|Q| \|h_1\|_{\mathcal Y_r(\mathcal Q)} \|h_2\|_{\mathcal Y_{r'}(\mathcal Q)}\]
uniformly over all $\mu,\nu,$ all  dyadic cubes  $Q$ and stopping collections  $\mathcal Q $ with top $  Q$.  
\end{lemma}
\begin{proof}  We may estimate the first term in the definition \eqref{therealLambdaQ} as follows:
\begin{equation} \label{LambdaQsplit1}
|\Lambda_{Q,\mu,\nu}(h_1 ,h_2)| \leq C_{\mathsf{T}}(r) \|h_1\cic{1}_{Q}\|_{r}\|h_2\cic{1}_{3Q}\|_{r'} \lesssim    C_{\mathsf{T}}(r) |Q|\|h_1\|_{\mathcal Y_r} \|h_2\|_{\mathcal Y_{r'}}.
\end{equation}
Further,  using the support condition in \eqref{therealLambdaq} with $L$ in place of $Q$ and  the  disjointness property \eqref{stopprop} in the last step, we obtain
\[\begin{split}  
\sum_{L\in \mathcal Q: L\subset Q}  |\Lambda_{   L,\mu,\nu}(h_1 ,h_2)|& = \sum_{L\in \mathcal Q: L\subset Q}  |\Lambda_{ L,\mu,\nu}(h_1\cic{1}_{L} ,h_2\cic{1}_{3L})| \leq  C_{\mathsf{T}}(r)\sum_{L\in \mathcal Q: L\subset Q} \|h_1\cic{1}_{L}\|_{r}\|h_2\cic{1}_{3L}\|_{r'}   \\ &   \lesssim  C_{\mathsf{T}}(r) \|h_1\|_{\mathcal Y_r} \|h_2\|_{\mathcal Y_{r'}} \sum_{L\in \mathcal Q} |L| \lesssim   C_{\mathsf{T}}(r) |Q| \|h_1\|_{\mathcal Y_r} \|h_2\|_{\mathcal Y_{r'}}.
\end{split}
\]
The proof of the lemma is thus completed by combining \eqref{LambdaQsplit1} with the last display. 
\end{proof}

Our main theorem hinges upon estimates which are modified versions of the one occurring in Lemma \ref{lemma0}, when one of the two arguments of $\Lambda_{\mathcal Q,\mu,\nu}$ belongs to $\mathcal X$-type localized spaces.
\begin{theorem}
\label{theoremABS} There exists a  positive absolute constant $\Theta$ such that the following holds.
Let $\Lambda$ be a bilinear form satisfying    \eqref{assker2} and \eqref{asstrunc} above.
Assume that   there exist $1\leq p_1,p_2<\infty$ and a positive constant $C_{\mathsf{L}}$ such that  the estimates
\begin{equation} 
\label{lest} \tag{$\mathsf{L}$}
\begin{split}
&\left| \Lambda_{ \mathcal Q,\mu,\nu}(b,h)\right|   \leq C_{\mathsf{L}}   |Q| \|b\|_{\dot{\mathcal X}_{p_1}(\mathcal Q)}   \|h\|_{\mathcal Y_{p_2}(\mathcal Q) }, \\
& \left| \Lambda_{ \mathcal Q,\mu,\nu}(h,b)\right|   \leq C_{\mathsf{L}} |Q| \|h\|_{{\mathcal Y}_{\infty}(\mathcal Q)}   \|b\|_{\dot{\mathcal X}_{p_2}(\mathcal Q) }, 
\end{split} \end{equation}
hold uniformly over all $\mu,\nu\in\mathbb Z$, all dyadic lattices $\mathcal D$, all $Q\in \mathcal D$ and all stopping collections $\mathcal Q\subset \mathcal D$ with top $Q$. 
 Then  the estimate
\begin{equation}
\label{thABSest}
\sup_{\mu,\nu\in \mathbb Z}\left|\Lambda_\mu^\nu(f_1,f_2)\right| \leq 2^{\Theta d} \Big[
C_{\mathsf{T}}(r) +C_{\mathsf{L}}  \Big] \sup_{\mathcal S}\mathsf{PSF}_{\mathcal S;{p}_1,{p}_2}(f_1, f_2), 
\end{equation}
  holds for all $f_j\in   L^{{p}_j}(\R^d)$ with compact support, $j=1,2$.
\end{theorem}
\begin{remark} \label{remABS}      By the  limiting argument of Remark \ref{remlimarg1}, the conclusion    \eqref{thABSest}  entails that
\begin{equation}
\label{thABSest2}
\left|\Lambda(f_1,f_2)\right| \leq 2^{\Theta d} \Big[
C_{\mathsf{T}}(r) +C_{\mathsf{L}}  \Big] \sup_{\mathcal S}\mathsf{PSF}_{\mathcal S;{p}_1,{p}_2}(f_1, f_2)
\end{equation}
when $f_1,f_2\in L^{\infty}(\R^d) $ with  compact support (say). If we know that $\Lambda$ extends boundedly to $L^{t}(\R^d)\times   L^{t'}(\R^d)$ for some $1<t<\infty$, another simple limiting argument using the dominated convergence theorem extends \eqref{thABSest2} to all $f_1\in L^{t}(\R^d)$, $f_2\in L^{t'}(\R^d)$. It is in this last form that Theorem \ref{theoremABS} will be applied to deduce Theorems \ref{theoremRH} and \ref{theoremBR}.
\end{remark}
\begin{remark}[A comparison between sparse domination principles]\label{Remphy} Theorem \ref{theoremABS} identifies rather clearly  the  conditions needed for sparse domination of a  kernel operator $T$, namely  the adjoint of the bilinear form $\Lambda$.   
Condition  \eqref{lest} is a localized reformulation of the \emph{constant-mean zero} cancellation around which $L^p,p\neq 2$ Calder\'on-Zygmund theory revolves, and it is essentially a strengthening of the weak-$L^{p_j}$ estimate for    $T$ ($j=1$)  and its adjoint ($j=2$).
Further, our assumption of uniform $L^r$-boundedness of   the   truncations in        \eqref{asstrunc} is much tamer than requiring $L^r$-boundedness of   the maximal truncations of $T$. In fact, our {theorem can be applied  even when no estimates for maximal truncations of $T$ are known}.

Of course the exponents  $p_j$ enter   the sparse domination estimate \eqref{thABSest}, while    the exponent $r$  occurring in   \eqref{asstrunc} does not.
  This is in contrast with the other sparse domination principles occurring in the literature. For instance, in \cite[Theorem 4.2]{Ler2015},   a sparse domination of type \eqref{sparsegen} with exponents $(r,1)$ is obtained for operators $T$ whose \emph{grand maximal function}
\[
 \mathcal M_T f(x):=\sup_{ Q\ni x} \sup_{y\in Q} \big|T\big(f\cic{1}_{\R^d\setminus 3Q}\big)(y)\big|
\]
has the weak-$L^r$ bound for some $r\geq 1$. Notice that $\mathcal M_T $ may be as large as the maximal truncation of $T$.

A further comparison can be drawn with the abstract result of \cite{BFP}, which is a sparse domination principle for non-integral singular   operators. The off-diagonal estimate assumption \cite[Theorem 1.1(b)]{BFP} is a clear counterpart of 
\eqref{assker2}, while the maximal truncation assumption \cite[Theorem 1.1(c)]{BFP} is the non-kernel analogue of the grand maximal function from \cite{Ler2015}. It would be interesting to investigate whether, in the non-kernel setting of \cite{BFP}, an assumption in the vein of \eqref{lest} can be used instead.
\end{remark}

\begin{remark}[The essence of \eqref{lest}] 
Let $\mathcal{Q}$ be a stopping collection with top $Q$. When $b$ belongs to an $ \mathcal X_{\alpha}(\mathcal{Q})$-type space,  the forms
\[
  (b,h) \mapsto \Lambda_{ \mathcal{Q},\mu,\nu}(b,h), \qquad  (b,h)  \mapsto \Lambda_{ \mathcal{Q},\mu,\nu}(h,b)
\]
have a much more familiar representation, which is what allows  verification of assumption  \eqref{lest} in practice. By rephrasing the definition, when $b\in \mathcal X_{1} (\mathcal Q)$ is supported on $Q$ (which we can assume with no restriction) we have the equality
\begin{equation}
\label{Kok}
\Lambda_{ \mathcal{Q},\mu,\nu}(b,h)=  \sum_{j\geq 1}\int \sum_{  \mu <s\leq\min\{ s_{ Q},\nu\} }    K_{s}(x,y) b_{s-j}(y)  h(x)  \, \d y   \d x.  
\end{equation}
where
\[
b_s:=\sum_{\substack{
L\in \mathcal Q\\   s_L = s}}   b_L.
\]
This notation will be used throughout the paper: see for instance \eqref{Kokadj} below.
Furthermore, if $q$ is the exponent occurring in \eqref{assker2}, $h\in \mathcal Y_{q'} (\mathcal Q)$, and  $b\in \mathcal X_{q'} (\mathcal Q)$, then $\Lambda_{ \mathcal{Q},\mu,\nu}(h,b)$ is essentially self-adjoint up to a tolerable error term. Namely, if $h$ is supported on $Q$ (which we can also always assume), there holds
\begin{equation}
\label{Kokadj}
\Lambda_{ \mathcal{Q},\mu,\nu}(h, b)= \left( \sum_{j\geq 1} 
\int  \sum_{\mu <s\leq\min\{ s_{ Q},\nu\} }   {K_{s}(y,x)} b^{\mathsf{in}}_{s-j}(y)  h(x)  \, \d y   \d x\right)+ V_{ \mathcal Q }(h, b) 
\end{equation}
where 
\[
b^{\mathsf{in}} = \sum_{\substack{L\in \mathcal Q\\ 3L \cap 2Q \neq \emptyset}} b_L
\] is a truncation of $b$ and thus also belongs to  $  \mathcal X_{q'} (\mathcal Q)$ with $\| b^{\mathsf{in}}\|_{{\mathcal X}_{q'}(\mathcal Q) }\leq \|b\|_{{\mathcal X}_{q'}(\mathcal Q) }$,
and the remainder $V_{ \mathcal Q }(h, b) $ satisfies
\begin{equation} \label{GNI}  \begin{split}
|V_{ \mathcal Q }(h, b)|& \leq 2^{\vartheta d} [K]_{0,q} |Q| \|h\|_{{\mathcal Y}_{q'}(\mathcal Q)}   \|b\|_{{\mathcal X}_{q'}(\mathcal Q) }\end{split}
\end{equation}
for a suitable positive absolute constant $\vartheta$.
The representation \eqref{Kokadj}-\eqref{GNI} is a simple consequence of the structure of $b\in \mathcal X_{q'}(\mathcal{Q})$ and of the separation properties \eqref{stopprop}, \eqref{separation}.
We provide the necessary details for \eqref{Kokadj}-\eqref{GNI} in Appendix \ref{ssadj} at the end.
\end{remark}

\subsection{Proof of Theorem \ref{theoremABS}}  
 Given a form $\Lambda$   satisfying the assumptions of Theorem \ref{theoremABS}, $\mu<\nu \in \mathbb Z$
 and $f_j\in   L^{{p}_j}(\R^d)$, $j=1,2$, with  compact support, we will construct a sparse collection $\mathcal S$ of cubes of $\R^d$ such that    \begin{equation} 
\label{sparse-est}
  \big|\Lambda_\mu^\nu(f_1 ,f_2)\big| \leq 2^{\Theta d}\mathbf{C}\sum_{Q\in \mathcal S} |Q| \langle f_1 \rangle_{{p}_1,Q} \langle f_2 \rangle_{{p}_2,Q}
\end{equation}
  where $\mathbf{C}$ is   the expression within the square brackets in the conclusion of Theorem \ref{theoremABS}.
Here and below, we denote by $\Theta$ a suitably large positive absolute constant which will be chosen during the course of the proof. Within this proof,  we will also denote by $\vartheta$  positive absolute   constants which belong to $[2^{-8}\Theta, 2^{-7}\Theta]$    and may differ at each occurrence.   
As the assumptions of Theorem \ref{theoremABS} are stable if we replace $\Lambda$ with $\Lambda_\mu^\nu$, we can work under the assumption that that $K_s=0$ for all $s\notin (\mu,\nu]$ and thus drop $\mu,\nu$ from the notations \eqref{therealLambdaq}, \eqref{therealLambdaQ}. 

The proof of \eqref{sparse-est} is iterative and is carried out in Subsection \ref{ssiter} below. Here, we enucleate the main estimate for the form $\Lambda^{s_Q}$ from \eqref{therealLambdaq}  in terms of stopping collection norms.
\begin{lemma}\label{iterlemma} Let   $Q$ be a fixed dyadic cube in $\R^d$ and $\mathcal Q$ be a stopping collection with top $Q$. Then
\begin{equation}
\label{iterativeSC} 
 \left|\Lambda^{s_Q}(h_1\cic{1}_Q,h_2\cic{1}_{3Q})\right| \leq  2^{\vartheta d} \mathbf{C}  |Q| \|h_1\|_{ \mathcal{Y}_{{p}_1}(\mathcal Q)}\|h_2\|_{ \mathcal{Y}_{{p}_2}(\mathcal  Q)}  +   \sum_{{\substack{L\in \mathcal Q\\ L\subset Q}} }   \left| \Lambda^{s_L} (h_1\cic{1}_L,h_2\cic{1}_{3L})  \right|\end{equation}
 \end{lemma}
 \begin{proof} We are free to assume that $\supp h_1\subset Q, \supp h_2\subset 3Q$ for simplicity of notation. 
For $j=1,2$, construct the Calder\'on-Zygmund decomposition of $h_j$ with respect to the family $\mathcal{Q}$ as described in Remark \ref{CZrem}, that is 
\[
h_j = g_j+ b_j, \qquad   b_j=\sum_{L\in \mathcal{Q}} b_{j L}, \qquad b_{jL}:=\left( h_j- { \frac{1}{|L|} \int_{L} h_j(x) \, \d x}\right)\cic{1}_L,
\]
The Calder\'on-Zygmund properties in this context are, for $j=1,2$,
\[
\|g_j\|_{\mathcal{Y}_\infty} \lesssim \|h_j\|_{ \mathcal{Y}_{{p}_j} }, \qquad \| b_j\|_{\dot{\mathcal X}_{p_j} } \lesssim   \|h_j\|_{ \mathcal{Y}_{{p}_j}}.\]
Using the definition \eqref{therealLambdaQ}, we decompose on our way to \eqref{iterativeSC}
\begin{equation}\label{splitting}
\begin{split}
&\quad \Lambda^{s_{Q}}(h_1, h_2)  \\
 &=  {\Lambda}_{ \mathcal Q}(h_1, h_2) +  \sum_{{\substack{L\in \mathcal Q\\ L\subset Q}} } \Lambda^{s_L}(h_1\cic{1}_L,h_2)       \\
 &=     \Lambda_{ \mathcal Q}(g_1, g_2)   + 
\Lambda_{ \mathcal Q}(b_1, g_2)   + 
\Lambda_{ \mathcal Q}(g_1, b_2) + 
\Lambda_{ \mathcal Q}(b_1, b_2 ) +\sum_{{\substack{L\in \mathcal Q\\ L\subset Q}} }  \Lambda^{s_L}(h_1\cic{1}_L,h_2\cic{1}_{3L})   
\end{split}
\end{equation}
 The last  sum on the last right hand side is estimated by the sum appearing on the right hand side of \eqref{iterativeSC}.   
We are left with estimating the first four  terms in  the last line of \eqref{splitting}.
 The leftmost is controlled by the estimate of Lemma \ref{lemma0}:
 \[
\begin{split} |\Lambda_{ \mathcal Q}(g_1, g_2) |&  \lesssim C_{\mathsf{T}}(r)   |Q| \|g_1\|_{\mathcal Y_r} \|g_2\|_{\mathcal Y_{r'}  }  \lesssim \mathbf{C}|Q|\|h_1\|_{ \mathcal{Y}_{{p}_1}} \|h_2\|_{ \mathcal{Y}_{{p}_2}}.
\end{split} 
\]
The second term is handled by appealing to assumption $(\mathsf L)$, which yields \[ 
\begin{split}
| \Lambda_{ \mathcal Q}(b_1, g_2)| &\leq  
 C_{\mathsf{L}}   |Q| \|b_1\|_{\dot{\mathcal X}_{p_1} }  \|g_2\|_{\mathcal Y_{p_2} }  \lesssim \mathbf{C}  |Q| \|h_1\|_{ \mathcal{Y}_{{p}_1}} \|h_2\|_{ \mathcal{Y}_{{p}_2}}
 \end{split}  
\]
where the second estimate follows from   the Calder\'on-Zygmund properties above.
The third is also  estimated by appealing to  $(\mathsf L)$, as
\[
\begin{split}
|\Lambda_{ \mathcal Q}(g_1, b_2)| &\leq  
 C_{\mathsf{L}}  |Q| \|g_1\|_{\mathcal Y_\infty }  \|b_2\|_{\dot{\mathcal X}_{p_2} }  \lesssim \mathbf{C}|Q|  \|h_1\|_{ \mathcal{Y}_{{p}_1}} \|h_2\|_{ \mathcal{Y}_{{p}_2}}.
 \end{split}  
\]
Finally, again by assumption $(\mathsf L)$,
\[ \begin{split} 
|\Lambda_{ \mathcal Q}(b_1, b_2)|& \leq
 C_{\mathsf{L}}  |Q_0| \|b_1\|_{\dot{\mathcal X}_{p_1} }  \|b_2\|_{\mathcal Y_{p_2} }    \lesssim   \mathbf{C}   |Q| \|h_1\|_{ \mathcal{Y}_{{p}_1}} \|h_2\|_{ \mathcal{Y}_{{p}_2}},\end{split}
\]
where the final inequality follows again from   the Calder\'on-Zygmund estimates. The proof of Lemma \ref{iterlemma} is thus complete. \end{proof}

\subsection{Proof of \eqref{sparse-est}} \label{ssiter} The proof is obtained by means of the iterative procedure described below.
\vskip1mm \noindent  \textsc{Preliminaries}.
We will produce stopping collections iteratively, by suitable Whitney decompositions of unions of  sets
\begin{equation}
\label{EQ} E_{Q} =\left\{x\in 3Q: \max_{j=1,2} \frac{\mathrm{M}_{{p}_j} (f_j\cic{1}_{3Q})(x)}{\l f_j \r_{{p}_j,3Q}}> 2^{\frac{\Theta d}{4} }\right\}
\end{equation}
associated to a cube $Q$ and a pair of functions $f_1,f_2$. We notice  that
\begin{equation}
\label{EQest}E_Q \subset 3Q, \qquad  |E_{Q}| \leq 2^{-\vartheta d} |Q|;
\end{equation}
the measure estimate is a consequence of   the maximal theorem, and holds provided $\Theta$ is chosen sufficiently large.
  In this proof, we say that two dyadic cubes $L,L'$ are \emph{neighbors}, and  write $L \sim L'$, if 
\[
7L\cap 7L' \neq \emptyset, \qquad |s_L-s_{L'}| < 8.
\]
The separation condition \eqref{separation} tells us that if the $7$-fold dilates  of two cubes  $L,L'$ belonging to the same stopping collection intersect nontrivially, then $L,L'$ must be neighbors.
We also recall the notation $\widehat L$ for the $2^{5}$-fold dilate of $L$.
\vskip1mm \noindent  \textsc{Initialize}: let $f_j\in   L^{{p}_j}(\R^d)$, $j=1,2$, with  compact support be fixed. By suitably choosing the dyadic lattice $\mathcal D$,  we may find $Q_0\in \mathcal D$ such that $\supp f_1 \subset Q_0$, $\supp f_2 \subset 3 Q_0$ and $s_{Q_0}$ is larger than the largest nonzero scale occurring in the kernel. Then set $\mathcal S_0=\{Q_0\}$, $E_0=3Q_0$, and define referring to \eqref{EQ} 
\[\begin{split}
& E_1:=E_{Q_0},\\ 
&\mathcal S_1:= \textrm{maximal   cubes } L \in \mathcal D \textrm{ such that } 9L \subset  E_{1}.
\end{split}
\]
 Notice that the following properties are satisfied:
\begin{align} \label{disj0} & L\in \mathcal S_1 \textrm{ are a pairwise disjoint collection},\\
\label{sparse0} &   E_1 = \bigcup_{L \in \mathcal S_1} L = \bigcup_{L \in \mathcal S_1} 9L\subset E_{0}, \qquad \left|Q_0\setminus E_1 \right|   \geq \left(1-2^{-d\vartheta}\right)|Q_0|, \\
\label{neighbor0} &   L,L'\in \mathcal S_1, \,7L\cap 7L'\neq \emptyset\implies L \sim L' . \end{align}
 Properties 
 \eqref{disj0} and the first part of \eqref{sparse0} are by construction, while the second part of \eqref{sparse0} follows from the estimate of \eqref{EQest}. For \eqref{neighbor0}  suppose instead that $7L \cap 7L'$ is not empty when $s_L\leq s_{L'}-8$. By the relation between the sidelengths it follows  that $\widehat L \subset 9L'$, which implies that the $9$-fold dilate of the dyadic parent of $L$ is contained in $9L'$ as well,   contradicting the maximality of $L$. By virtue of \eqref{disj0}--\eqref{neighbor0}, $\mathcal Q_1(Q_0):=\mathcal S_1$ is a stopping collection with top $Q_0$; compare with \eqref{stopprop}, \eqref{separation}.
The first property in \eqref{sparse0} guarantees that 
\[
\sup_{x \not \in \mathsf{sh} \mathcal{ Q}_{1} (Q_0)} |f_j(x)| \leq 2^{\frac{\Theta d}{4}} \l f_j\r_{{p}_j,3Q_0}. 
\]
Further, by the maximality condition on $L\in \mathcal S_1$, it follows that
\[
\sup_{L \in \mathcal Q_1(Q_0)} \inf_{\widehat L}\,  \mathrm{M}_{{p}_j} (f_j\cic{1}_{3Q_0})  \leq 2^{\frac{\Theta d}{4}} \l f_j\r_{{p}_j,3Q_0} 
\]
for $j=1,2$. The last two inequalities tell us that 
\[
\|f_j\|_{\mathcal Y_{{p}_j}(\mathcal Q_1(Q_0))} \leq 2^{\frac{\Theta d}{4}} \l f_j\r_{{p}_j,3Q_0}, \qquad j=1,2.
\]
Applying \eqref{iterativeSC} to the stopping collection  $\mathcal Q_1(Q_0)$, and $h_1=f_1,h_2=f_2$  we obtain
\begin{equation}
\label{iterative2SC0st}
\begin{split} & \quad  \big|\Lambda(f_1 ,f_2)\big|=  
   \big|\Lambda^{s_{Q_0}}(f_1\cic{1}_{Q_0} ,f_2\cic{1}_{3Q_0})\big| \\ & \leq 2^{\Theta d}\mathbf{C} |Q_0| \langle f_1 \rangle_{{p}_1,3Q_0} \langle f_2 \rangle_{{p}_2,3 Q_0} +
    \sum_{\substack{L \in {\mathcal Q}_1(Q_0)\\ L\subset Q_0}}   \left|\Lambda^{s_{L}} (f_1\cic{1}_L,f_2\cic{1}_{3L})  \right|.  \end{split}\end{equation}
The obtained properties   \eqref{disj0}--\eqref{neighbor0} and estimate \eqref{iterative2SC0st} are the $\ell=1$ case of the induction assumption in the inductive step below.  
\vskip1mm \noindent
\textsc{Inductive step}: Suppose inductively collections $\mathcal S_\ell$, $0\leq \ell\leq k$ and sets $E_\ell$, $1\leq \ell \leq k$  have been constructed, with the properties that for all $1\leq \ell \leq k$ 
\begin{align}
\label{disjk} & L\in \mathcal S_\ell \textrm{ are a pairwise disjoint collection},\\
\label{sparsek} 
& E_\ell = \bigcup_{L \in \mathcal S_\ell} L = \bigcup_{L \in \mathcal S_\ell}9 L  \subset E_{\ell-1}, \qquad |Q\setminus E_\ell|\geq \left(1-2^{-\vartheta d}\right) |Q| \quad \forall Q \in \mathcal S_{\ell-1},    \\
\label{neighbork} &   L,L'\in \mathcal S_\ell,\; 7L\cap 7L'\neq \emptyset\implies L\sim L'  .  
\end{align}
Suppose also that if $\mathcal T_{k-1}=\mathcal S_0\cup\cdots \cup \mathcal S_{k-1}$, the estimate
\begin{equation}
\label{iterative2SC1} 
  \big|\Lambda(f_1 ,f_2)\big| \leq 2^{\Theta d} \mathbf{C} \sum_{R\in \mathcal T_{k-1}}|R| \langle f_1 \rangle_{{p}_1,3R} \langle f_2 \rangle_{{p}_2,3R} +
     \sum_{Q \in {\mathcal S}_k} \left|\Lambda^{s_{Q}} (f_1\cic{1}_Q,f_2\cic{1}_{3Q}) \right| \end{equation}
has been shown to hold.   
At this point define
\[
\begin{split}
&E_{k+1}:=\bigcup_{Q\in \mathcal S_k} E_{Q},  
\\
&\mathcal S_{k+1}:= \textrm{maximal  cubes } L \in \mathcal D \textrm{ such that } 9L \subset  E_{k+1},\\ 
& \mathcal Q_{k+1} (Q) =\{L \in \mathcal S_{k+1}: L\subset 3Q \}, \qquad Q\in \mathcal S_k.
\end{split}
\]
Property \eqref{disjk}, together with the first property in \eqref{sparsek}, as $E_Q\subset 3Q\subset E_{k}$, and  \eqref{neighbork}, via  the same reasoning we used for \eqref{neighbor0}, now hold for $\ell=k+1$ as well. Let now $Q\in \mathcal S_k$. Property \eqref{neighbork} with $\ell=k$ implies that 
\[
3Q\cap E_{k+1}\subset \bigcup_{Q'\in \mathcal S_k: Q'\sim Q} E_{Q'}.
\]Therefore, we learn that
\begin{equation}\label{sparsek1}
|Q\cap E_{k+1}| \leq |3Q\cap E_{k+1}|\leq\sum_{ Q'\in \mathcal S_k: Q'\sim Q} |E_{Q'}| \leq 2^{-\vartheta d }|Q| 
\end{equation}
by applying for each $Q'\in \mathcal S_k$ with  $ Q'\sim Q$ the  estimate of \eqref{EQest}, and observing  that the cardinality of $\{Q'
\in \mathcal{D}:Q'\sim Q\}$ is bounded by an absolute dimensional constant, and  $|Q|, |Q'|$ are comparable,  again up to an absolute dimensional constant. From the above display we obtain the second part of \eqref{sparsek} for $\ell=k+1$. Moreover, one observes that if $L\in\mathcal S_{k+1}$ with $L\cap 3Q\neq \emptyset$, then by virtue of property \eqref{sparsek1}, $L$ must be significantly shorter than $Q$ and thus contained in one of   the $3^d$ translates of the dyadic cube $Q$ whose union covers $3Q$. Namely, we have the equality
\[
\mathcal Q_{k+1} (Q)=\{L\in\mathcal S_{k+1}: L\cap 3Q\neq \emptyset\}
\]
which also entails the last equality in 
\[
 \bigcup_{L\in   \mathcal Q_{k+1} (Q):3L\cap  2Q\neq \emptyset} 9L   \subset \bigcup_{L\in \mathcal  S_{k+1}:L\cap 3Q\neq \emptyset} L = \bigcup_{L\in\mathcal Q_{k+1} (Q)} L = \mathsf{sh} \mathcal{ Q}_{k+1} (Q)
\]
as the set in the first left hand side of the last display is contained in $3Q$ and \eqref{sparsek} holds for $\ell=k+1$.
Comparing with \eqref{stopprop}, \eqref{separation}, the discussion above entails that  $\mathcal Q_{k+1} (Q)$ is a stopping collection with top $Q$ and such that  $E_Q \subset \mathsf{sh} \mathcal{ Q}_{k+1} (Q)$, so that
\[
\sup_{x \not \in  \mathsf{sh} \mathcal{ Q}_{k+1} (Q) } |f_j\cic{1}_{3Q}(x)| \leq 2^{\frac{\Theta d}{4}} \l f_j\r_{{p}_j,3Q}. 
\]
 Furthermore, for $j=1,2$
\[
\sup_{L \in \mathcal Q_{k+1}(Q) } \inf_{\widehat L}\,  \mathrm{M}_{{p}_j} (f_j\cic{1}_{3Q})  \leq 2^{\frac{\Theta d}{4}} \l f_j\r_{{p}_j,3Q},
\]
otherwise the $9$-fold dilate of the dyadic parent of some $ L \in \mathcal Q_{k+1}(Q)  $ would be contained in $E_Q$ and thus in $E_{k+1}$, contradicting the maximality of such $L$. Therefore 
\[
\|f_j\cic{1}_{3Q}\|_{\mathcal Y_{{p}_j}(\mathcal Q_{k+1}(Q))} \leq 2^{\frac{\Theta d}{4}} \l f_j\r_{{p}_j,3Q}, \qquad j=1,2,
\]
and we may  apply \eqref{iterativeSC} to each $Q\in \mathcal S_k$ summand in \eqref{iterative2SC1}, with $h_1=f_1 $, $h_2=f_2 $ and obtain 
\[\begin{split}
  \big|\Lambda^{s_{Q}}(f_1\cic{1}_{Q} ,f_2\cic{1}_{3Q})\big| &\leq 2^{\Theta d} \mathbf{C} \ |Q| \langle f_1 \rangle_{{p}_1,3Q} \langle f_2 \rangle_{{p}_2,3Q} +
    \sum_{L \in {\mathcal Q}_{k+1}(Q):L\subset Q} \left|\Lambda^{s_{L}} (f_1 \cic{1}_L,f_2\cic{1}_{3L}) \right| \\   &=2^{\Theta d} \mathbf{C} \ |Q| \langle f_1 \rangle_{{p}_1,3Q} \langle f_2 \rangle_{{p}_2,3Q} +
    \sum_{L \in {\mathcal S}_{k+1}: L\subset Q} \left|\Lambda^{s_{L}} (f_1\cic{1}_L,f_2\cic{1}_{3L}) \right|. \end{split}
\]
As $Q\in \mathcal S_k$ are pairwise disjoint, see \eqref{disjk}, summing   over $Q\in \mathcal S_k$, writing $\mathcal T_{k}=\mathcal S_0\cup\cdots \cup \mathcal S_{k}$ and combining the resulting estimate with \eqref{iterative2SC1}, we arrive at \[
   \big|\Lambda(f_1 ,f_2)\big| \leq 2^{\Theta d} \mathbf{C} \sum_{Q\in \mathcal   T_{k}}|Q| \langle f_1 \rangle_{{p}_1,3Q} \langle f_2 \rangle_{{p}_2,3Q} +
  \sum_{L \in {\mathcal S}_{k+1}} \left| \Lambda^{s_{L}} (f_1\cic{1}_L,f_2\cic{1}_{3L}) \right| \]
that is, \eqref{iterative2SC1} with $k$ replaced by $k+1$. This, together with the previously obtained \eqref{disjk}, \eqref{sparsek} and \eqref{neighbork} for $\ell=k+1$, completes the current iteration. 
\vskip1mm \noindent  \textsc{Termination}: a consequence of our construction is  that $\sigma_k:=\max\{s_Q: Q \in \mathcal S_k \}\leq s_{Q_0}-\vartheta k.$ The   algorithm terminates when $k=K$, where $K$ is such that $\sigma_K$ is strictly less than the minimal nonzero scale in the kernel.  For $k=K$ in \eqref{iterative2SC1} the second sum on the right hand side vanishes identically and we have obtained the   estimate \eqref{sparse-est} by setting $\mathcal T:=\mathcal T_{K-1}$ and $\mathcal S:=\{3Q: Q\in \mathcal T\}$.
We see that the collection $\mathcal T$, and thus the collection of the dilates $\mathcal S$, are  sparse by simply observing that the sets
\[
F_Q:= Q\setminus E_{k+1}, \qquad Q \in \mathcal S_{k} \]  
are pairwise disjoint for $Q\in \mathcal T$  and have   measure larger than $(1-2^{-d\vartheta})|Q|$, as  can be seen from \eqref{sparsek}.

\section{Localized estimates for Dini and H\"ormander-type kernels} \label{secDH}   
In the first part of this section, we state and prove a family of localized estimates, of the type occurring in condition \eqref{lest} of Theorem \ref{theoremABS}, for kernels falling within the scope of \eqref{assker2} and possessing additional smoothness properties, of Dini or H\"ormander type. These estimates and their proof are a reformulation of the classical inequalities intervening in the proof of the weak-$L^1$ bound for Calder\'on-Zygmund operators (see, for example, \cite[Chapter 1]{Stein}). We choose to provide   details as we believe the arguments to be rather explanatory of the driving philosophy behind Theorem \ref{theoremABS}.

As we mentioned in the introduction, our abstract Theorem \ref{theoremABS}, coupled with the localized estimates that follow, can be employed to reprove the optimal  sparse domination estimates for   Cald\'eron-Zygmund kernels of Dini and H\"ormander type, thus recovering the results (among others) of \cite{BCDHL,HRT,Lac2015,Ler2015,Li1}. We provide a summary of the statements of such domination theorems in the second part of this section. 

\subsection{Localized estimates and kernel norms}
Throughout these estimates, we assume that a stopping collection $\mathcal Q$ with top $Q$ as in Section \ref{secABS}  has been fixed, and the notations   $\Lambda_{\mathcal Q,\mu,\nu}$  refer to \eqref{therealLambdaQ}. It is understood that the   constants implied by the almost inequality signs depend on dimension only and are in particular are uniform  over the choice of $\mathcal Q$.
 We begin with the  single scale localized estimate where no cancellation is exploited.\begin{lemma}[Trivial estimate]
Let  $1< \beta \leq \infty$ and $\alpha=\beta'$. Then for all $j\geq 1$ there holds \label{trivialestlemma} 
\[
\sum_{s}\int |K_s(x,y)| |b_{s-j}(y)| |h(x)|\, \d y \d x \lesssim    [K]_{0,\beta} |Q| \|b\|_{\mathcal X_1}  \|h\|_{\mathcal Y_\alpha }.
\]
\end{lemma}
\begin{proof} 
As $\|b_L\|_{1}\lesssim |L|\|b\|_{\mathcal X_1} $ for $L \in \mathcal Q$, it suffices to prove that for each  $L\in \mathcal Q $ and $s=s_L+j$ there holds
\begin{equation}
\label{tailsest}
\int |K_{s}(x,y)| |b_{L}(y)| |h(x)|\, \d y \d x \lesssim  [K]_{0,\beta}  \|b_L\|_{1}  \|h\|_{\mathcal Y_\alpha}.
\end{equation}
In turn, it then suffices to prove that  
\[
s\geq s_L\implies \sup_{y\in L}
 \int |K_s(y+u,y)|  |h(y+u)|\, \d u \lesssim   [K]_{0,\beta}  \|h\|_{\mathcal Y_\alpha }  
\]
which readily follows from
\[
\begin{split}
 \int |K_s(y+u,y)|  |h(y+u)|\, \d u & \leq \|K_s(y+\cdot,y)\|_{\beta} \left(\int_{B(y,2^{s+10})} |h(z)|^\alpha\, \d z\right)^{\frac{1}{\alpha}}  \\ & \lesssim   [K]_{0,\beta}   \left( \inf_{  \widehat L} \mathrm{M}_\alpha h \right) \leq [K]_{0,\beta}  \|h\|_{\mathcal Y_\alpha } 
\end{split} 
 \] when $y\in L$.
Above, we used the support condition \eqref{assker2} and H\"older's inequality for the first step, and subsequently that the ball $B(y,2^{s+10})=\{z\in \R^d:|z-y|<2^{s+10}\}$ contains the dilate $\widehat L$. The proof is complete.
\end{proof}
We  introduce a further family of   kernel norms in addition to the one of \eqref{assker2}, to which we refer for notation.  For $1< \beta\leq \infty$ set 
\begin{equation}
\label{NORMS}
[K]_{1,\beta}:=\sum_{j=1}^\infty   \varpi_{j,\beta}(K)
\end{equation}
where 
\[
\varpi_{j,\beta}(K) := \sup_{s\in \mathbb Z} 2^{\frac{sd}{\beta'}}\sup_{x\in \R^d  }\sup_{\substack{h\in \R^d\\ \|h\|_\infty< 2^{s-j-1}}}\left(\begin{array}{ll}
  & \|K_s(x,x+\cdot)- K_s(x+h,x +\cdot)\|_{\beta} \\  +&   \|K_s(x+\cdot,x)-K_s(x+\cdot,x+h)\|_{\beta}\end{array} \right).
\]
The second localized estimate we consider uses the finiteness of $[K]_{1,\beta}$ to incorporate the constant-mean zero cancellation effect.
\begin{lemma}[Cancellation estimate] \label{bbestimate1} Let  $1< \beta \leq \infty$ and $\alpha=\beta'$. Then for all $\mu,\nu \in \mathbb Z$ there holds  
\begin{equation} \label{localestlemma}
\left| \Lambda_{\mathcal{Q},\mu,\nu}(b,h)\right| + \left|  \Lambda_{\mathcal{Q},\mu,\nu}(h,b)\right|\lesssim \left([K]_{0,\infty} +[K]_{1,\beta} \right) |Q| \|b\|_{\dot{\mathcal X}_1}  \|h\|_{\mathcal Y_\alpha }.
\end{equation}
\end{lemma}
\begin{proof} 
It will suffice to prove the estimate 
\begin{equation} \label{localestlemma2}
  \sum_{L\in \mathcal Q}   \sum_{j=1 }^\infty \left|\int  K_{s_L+j}(x,y) \widetilde{b}_{L}(y) \widetilde{h}(x)    \, \d y \, \d x\right|\lesssim  [K]_{1,\beta}  |Q| \|\widetilde b\|_{\dot{\mathcal X}_1}  \|\widetilde h\|_{\mathcal Y_\alpha }.\end{equation}
In fact, by using the representations in \eqref{Kok}, \eqref{Kokadj} we see that  for all $\mu,\nu\in \mathbb Z$ and each pair $b\in \dot{ \mathcal X}_1, h\in \mathcal Y_\alpha$, the forms     $| \Lambda_{\mathcal{Q},\mu,\nu}(b,h)|$,  $ |\Lambda_{\mathcal{Q},\mu,\nu}(h,b)|$ are both bounded above  by the left hand side of \eqref{localestlemma2}
for suitable $\widetilde{b}\in\dot{\mathcal{X}}_1,\widetilde{h}\in\mathcal{Y}_\alpha$ whose norms are dominated by $\|b\|_{\dot{\mathcal X}_1}, \|h\|_{\mathcal Y_\alpha }$ respectively, up to possibly replacing  $K_s$ with its transpose and controlling the remainder  term  $V_{ \mathcal Q }(h, b)$ in the case of $\Lambda_{\mathcal{Q},\mu,\nu}(h,b)$. This remainder is estimated in  \eqref{GNI} for $q=\infty$, which  is acceptable for the  right hand side of  \eqref{localestlemma}.

We will obtain estimate \eqref{localestlemma2} from the bound
\begin{equation} \label{localestlemma3}
\sum_{j=1 }^\infty \left|\int  K_{s_L+j}(x,y) \widetilde{b}_{L}(y) \widetilde{h}(x)    \, \d y \, \d x\right| \lesssim  [K]_{1,\beta}  |L| \|\widetilde b\|_{\dot{\mathcal X}_1}  \|\widetilde h\|_{\mathcal Y_\alpha },\qquad L\in \mathcal Q
\end{equation}
 by summing over $L\in \mathcal Q$ in and using their disjointness \eqref{stopprop}.
Fix  $L\in \mathcal Q$ and $j\geq 1$.  Using the cancellation of $\widetilde{b}_{L}$ and then arguing as in the proof of \eqref{tailsest} above we obtain
\[\begin{split}  
\left|\int  K_{s_L+j}(x,y) \widetilde{b}_{L}(y) \widetilde{h}(x)    \, \d y \, \d x\right| 
&\leq \|\widetilde{b}_L\|_{1} \sup_{y\in L} \int |K_{s_L+j}(y+u,y)-K_{s_L+j}(y+u,c_L)|  |\widetilde{h}(y+u)|\, \d u  \\&\lesssim \|\widetilde{b}_L\|_{1} \omega_{j,\beta}(K) \left( \inf_{\widehat{L}} \mathrm{M}_\alpha\widetilde{h} \right) \lesssim \omega_{j,\beta}(K) |L|\|\widetilde{b} \|_{\dot{\mathcal X}_1}\|\widetilde{h}\|_{\mathcal Y_\alpha}   \end{split}
\] 
and \eqref{localestlemma3}   follows by summing over  $j\geq 1$.\end{proof}

\subsection{Sparse domination of Calder\'on-Zygmund kernels}
\label{ssDH} We briefly mention how our abstract Theorem \ref{theoremABS} can be employed to recover sparse domination, and thus weighted bounds, for Calder\'on-Zygmund kernels with minimal smoothness assumptions.
Let  $T$ be an $L^2(\R^d)$ bounded operator whose   kernel $K$   satisfies the usual size normalization
\[
\sup_{x\neq y} |x-y|^d|K(x,y)| \leq 1.
\]
Let $\psi$ be a fixed Schwartz function supported in $A_1=\{x\in \R^d:2^{-2}<|x|<1\}$ and such that
\[
\sum_{s\in \mathbb Z} \psi(2^{-s} x) =1, \qquad x\neq 0.
\] 
It is immediate to see that  \eqref{assker2} holds, and in particular $[K]_{0,\infty}\leq C$, for the decomposition
\[
K_s(x,y):= K(x,y)\psi\left (\textstyle\frac{x-y}{2^s}\right), \qquad s \in \mathbb Z.
\]
We further assume that $[K]_{1,\beta}<\infty$ for some $1<\beta\leq \infty$, where the kernel norm has been defined in \eqref{NORMS}.
When $\beta=\infty$, this is exactly the Dini condition \cite{HRT,Lac2015,Ler2015}. For $\beta<\infty$, the above condition is equivalent to the assumptions of \cite{Li1}, where in fact a multilinear version is presented. 

The assumptions of Theorem \ref{theoremABS} then hold for the dual form \[\Lambda(f_1,f_2)=\l Tf_1,\overline{f_2}\r.\] We have already observed that \eqref{assker2}  is verified with $q=\infty$. It is well-known that $L^2$-bounded\-ness of $\Lambda$ together with $[K]_{1,\beta}<\infty$ yields that  the truncation forms $\Lambda^\nu_\mu$ (cf. \eqref{ltf})  are uniformly bounded on $L^t(\R^d)\times L^{t'}(\R^d)$ \cite[Ch.\ I.7]{Stein} for all $1<t<\infty$, thus we have condition \eqref{asstrunc} with, for instance,  $r=2$.  Furthermore, Lemma \ref{bbestimate1}  is exactly \eqref{lest} for the corresponding $\Lambda_{\mathcal Q,\mu,\nu }$, with $p_1=1,p_2=\alpha=\beta'$. Applying Theorem \ref{theoremABS} in the form given in Remark \ref{remABS}, we obtain the following sparse domination result, which recovers (the dual form of) the domination theorems from the above mentioned references. We send to the same references for the sharp weighted norm inequalities that descend from this result.
\begin{theorem}[Calder\'on-Zygmund theory] Let $T$ be as above and   $1\leq \beta <\infty$. For all $1<t<\infty$ and   all pairs  $f_1\in L^{t} (\R^d), f_2\in L^{t'} (\R^d)$ there holds
 \[
 |\l T f_1, f_2 \r| \leq C_\beta [K]_{1,\beta}\sup_{\mathcal S} \mathsf{PSF}_{\mathcal S;1,\beta'}(f_1, f_2).
 \]
where $C_\beta$ is a positive constant depending on $\beta$ and on the dimension $d$ only.
\end{theorem}

 \section{Proof of Theorem \ref{theoremRH}} \label{secRH}   Let $1<q\leq \infty$ and suppose that $\Omega\in L^q(S^{d-1})$ has unit norm and  vanishing integral.   Write throughout $x'=x/|x|$.  We decompose for $x\neq 0$ the kernel of $T_\Omega$ in \eqref{tomega} as
\[
\frac{\Omega(x')}{|x|^d} =  \sum_{s} K_{s}(x), \qquad K_{s}(x)=\Omega(x') 2^{-sd}\phi(2^{-s} x)
\]
where $\phi$ is a suitable smooth radial function supported in $A_1=\{2^{-2} \leq |x|\leq 1\}$. 
The main result of this subsection is the following proposition: again, we assume that a stopping collection  $\mathcal Q$ with top the dyadic cube  $Q$ as in Section \ref{secABS} has been fixed and the notations $\mathcal Y_t$ and similar refer to that fixed setting.
\begin{proposition} \label{propseeger} Let $\Omega\in L^q(S^{d-1})$  of unit norm and vanishing integral.
Let $\{\eps_s\}\in \{-1,0,1\}^{\mathbb Z}$ be a choice of signs, $b \in \dot{\mathcal X}_1$ and define
\[
 \mathsf{K}(b,h):=\sum_{j\geq 1} \sum_{s} \eps_{s}\left\langle   K_{s}*  b_{s-j},\overline h\right\rangle
\]
where
\[
b_{s} = \sum_{\substack{L \in \mathcal Q\\ s_{L}=s}} b_{L}.
\]
 There exists an absolute  constant $C$, in particular uniform over all $\{\eps_s\}\in \{-1,0,1\}^{\mathbb Z}$ such that  
 \begin{equation}
\label{seegest}
\left| \mathsf{K}(b,h)\right| \leq   \frac{Cp}{p-1}  |Q| \|b\|_{\dot{\mathcal X_{1}} }  \|h\|_{\mathcal Y_{p} }\begin{cases} \|\Omega\|_{L^{q,1}\log L(S^{d-1})} & q<\infty, \;p\geq q' \\ \|\Omega\|_{L^\infty(S^{d-1})} & q=\infty, \; p>1. \end{cases}
\end{equation}
\end{proposition}   With the above proposition in hand, we may now give the proof of Theorem \ref{theoremRH}. The structural assumptions \eqref{assker2}, \eqref{asstrunc} of the abstract Theorem \ref{theoremABS} applied to the above decomposition of (the dual form of) $T_\Omega$  are respectively  verified with $q=q$ and  with $r=2$ (this is the classical $L^2$-boundedness of the  truncations of $T_\Omega$ \cite{CZ,GS2}).

 We still need to ve\-ri\-fy \eqref{lest} for the values $p_1=1$ and $p_2=p$ for each $p$ in the claimed range (depending on whether $q=\infty$ or not).
It is immediate from  the representations \eqref{Kok}  that in this setting 
$
\Lambda_{\mathcal Q,\mu,\nu}(b,h) = \mathsf{K}(b\cic{1}_Q,h)
$ for a suitable choice of signs $\{\eps_s\}$ depending on $\mu,\nu$. So Proposition \ref{propseeger} yields the first condition in \eqref{lest} with $p_1=1,p_2=p$.
On the other hand, we read from  \eqref{Kokadj} that $\Lambda_{\mathcal Q,\mu,\nu}(h,b) $  is equal to $\mathsf{K}(b^{\mathsf{in}},h\cic{1}_Q)$, again for a suitable choice of signs $\{\eps_s\}$ depending on $\mu,\nu$, up to   replacing $K_s$ by ${K_s(-\cdot)}$, and up to subtracting off the remainder term from \eqref{GNI}, which is estimated in this case by an absolute constant times
\[
|Q|\|h\|_{\mathcal Y_\infty}\|b\|_{\mathcal Y_{q'}} \leq |Q|\|h\|_{\mathcal Y_\infty}\|b\|_{\mathcal Y_{p}}
\]
 which is acceptable for the right hand side of the second condition in \eqref{lest} when $p_2=p$.
These considerations and another application of Proposition \ref{propseeger} finally yield    Theorem \ref{theoremRH}, via our abstract result in the form   described in Remark \ref{remABS}.

\subsection{Proof of Proposition \ref{propseeger}} Throughout this proof, $C$ is a positive absolute dimensional constant which may vary at each occurrence without explicit mention. We assume $\{\eps_s\}\in \{-1,0,1\}^\mathbb Z$ is given. For the sake of simplicity, we redefine $K_s:=\eps_sK_s$; it will be clear from the proof below that the signs of $K_s$ play no  role.  Fix  a positive integer $j$. For $\delta>0$ to be fixed at the end of the argument define
\begin{equation}
\label{splitomega}
O_j= \big\{\theta \in S^{d-1}: |\Omega(\theta)| >2^{\delta j}\big\}, \qquad \Omega_j=\Omega \cic{1}_{S^{d-1}\setminus O_j}, \qquad \Delta_j=\Omega \cic{1}_{  O_j}.
\end{equation}
We now decompose
\begin{equation}
\label{seegerdec1}
K_s=H_s^j +  V_s^j, \qquad H_s^j=K_{s}\cic{1}_{\supp \Omega_j}, \quad V_s^j= K_{s}\cic{1}_{ O_j}.
\end{equation}
 The first localized form we treat, namely the contribution of the unbounded part of $\Omega$, is dealt with by means of a trivial estimate.\begin{lemma} \label{lemmaV}
$ \displaystyle
\mathsf{V}^j(b,h):=\sum_{s} |\langle  V_{s}^j* b_{s-j},\overline h  \rangle| \leq C   \|\Delta_j\|_{q} |Q| \|b\|_{\mathcal X_1}\|h\|_{\mathcal Y_p},\qquad p\geq q'$. 
\end{lemma}
\begin{proof} It suffices of course to prove the estimate above with $q'$ in place of $p$. This is actually a particular case of Lemma \ref{trivialestlemma} applied with $K=\{V_s^j\}$ and $\beta=q$, as it is immediate to see that for this kernel one has $[K]_{0,q}\leq C  \|\Delta_j\|_{q}. $\end{proof}
The contribution of the bounded part of $K_s$ in \eqref{seegerdec1} is more delicate, and we postpone the proof of the following lemma to the next  Subsection \ref{sspflemmaH}.
\begin{lemma} There exist   absolute constants $C,c>0$ such that for all $1<p\leq \infty$ \label{lemmaH}
\[
\mathsf{H}^j(b,h):=\left|\sum_{s} \langle  H_{s}^j* b_{s-j}, \overline h  \rangle\right| \leq C 2^{-cj\frac{p-1}{p}}  \|\Omega_j\|_{\infty} |Q| \|b\|_{\dot{\mathcal X_1}}\|h\|_{\mathcal Y_p}.\]
\end{lemma}
We may now complete the proof of Proposition \ref{propseeger}. We assume $q<\infty$, the remaining case is actually simpler as $\mathsf{V}^j$ is identically zero. Our decomposition   \eqref{seegerdec1}  
yields that
\[
| \mathsf{K}(b,h)|\leq \sum_{j\geq 1} |\mathsf{H}^j(b,h)|  + \sum_{j\geq 1} |\mathsf{V}^j(b,h)|.   
\]
 Choosing  $\delta=c\frac{p-1}{2p}$ in \eqref{splitomega}  and using Lemma  \ref{lemmaH}, we estimate
\[\begin{split}
  \sum_{j\geq 1} |\mathsf{H}^j(b,h)|  & \leq C  |Q| \|b\|_{\dot{\mathcal X_1}}\|h\|_{\mathcal Y_p} \sum_{j\geq 1} 2^{-cj\frac{p-1}{p}}  \|\Omega_j\|_{\infty}
  \leq C  |Q| \|b\|_{\dot{\mathcal X_1}}\|h\|_{\mathcal Y_p} \sum_{j\geq 1} 2^{-cj\frac{p-1}{2p}}  
\\ &\leq \frac{Cp}{p-1}|Q| \|b\|_{\dot{\mathcal X_1}}\|h\|_{\mathcal Y_p}   \end{split}
\]which is smaller than  the right hand side of \eqref{seegest}.
Using Lemma \ref{lemmaV}, the latter sum involving $\mathsf{V}_j$ is then estimated by
\[  
\left(\sum_{j\geq 1} \|\Delta_j\|_q\right)|Q| \|b\|_{\mathcal X_{1} }  \|h\|_{\mathcal Y_{p} } 
 \leq \frac{Cp}{p-1} \|\Omega\|_{L^{q,1}\log L(S^{d-1})} |Q| \|b\|_{\mathcal X_{1} }  \|h\|_{\mathcal Y_{p} }
\]
which also complies with the right hand side of \eqref{seegest}; here we have used 
that 
\[
 \sum_{j\geq 1} \|\Delta_j\|_q  \leq \sum_{j\geq 1 } \sum_{k\geq j} 2^{\delta k} |O_{k}\setminus O_{k+1}|^\frac{1}{q} \leq  \sum_{k\geq 1 } k 2^{\delta k} |O_{k}\setminus O_{k+1}|^\frac{1}{q} \leq \frac{C}{\delta} \|\Omega\|_{L^{q,1}\log L(S^{d-1})}.
\]
 The proposition is thus proved up to establishing  Lemma  \ref{lemmaH}.

\subsection{Proof of Lemma \ref{lemmaH}} \label{sspflemmaH} Our first observation is actually another trivial estimate.
\begin{lemma} \label{lemmaHtrivial} There exists  $C>0$ such that
$ \displaystyle
|\mathsf{H}^j(b,h)|  \leq C  \|\Omega_j\|_{\infty} |Q| \|b\|_{\mathcal X_1}\|h\|_{\mathcal Y_1}.
$
\end{lemma} 
\begin{proof}   This is an application of Lemma \ref{trivialestlemma} to   $K=\{H_s^j\}$ with $\beta=\infty$, as it is immediate to see that for this kernel one has $[K]_{0,\infty}\leq C \|\Omega_j\|_{\infty}. $\end{proof}
The second step is  an estimate with decay, but involving $\mathcal Y_\infty$ norms.
\begin{lemma} There exist   $ C,c>0$ such that \label{lemmaHinfty}
$ \displaystyle
|\mathsf{H}^j(b,h)|  \leq C  2^{-cj}\|\Omega_j\|_{\infty} |Q| \|b\|_{\dot{\mathcal X_1}}\|h\|_{\mathcal Y_\infty}.
$
\end{lemma} 
Before the  proof of Lemma \ref{lemmaHinfty}, which is given in the final Subsection \ref{sspflemmaHinfty}, we observe that the estimate of Lemma \ref{lemmaH} is obtained by Riesz-Thorin (for instance) interpolation in $h$ of  the last two lemmata.

\subsection{Proof of Lemma \ref{lemmaHinfty}} \label{sspflemmaHinfty}
 The techniques of this Subsection  are an elaboration of the arguments  of \cite{Seeger}. In  particular Lemma \ref{Glemma} below is a stronger version of  \cite[Lemma 2.1]{Seeger} while Lemma \ref{Ulemma}  is essentially the dual form of \cite[Lemma 2.2]{Seeger}. 
 
We perform a further  decomposition  of $H_s^j$. Let 
 $\Xi=\{e_\nu\}$ be a maximal $2^{-j-10d}$-separated   set contained in $\supp \Omega_j$. We may partition $\supp \Omega_j $ in $\#\Xi\lesssim 2^{j(d-1)}$ subsets $E_{\nu}$ each containing $e_\nu$ and such that $\mathrm{diam}|E_\nu|\lesssim 2^{-j} $. Set
\[
H_{s\nu}^j (x) = H_{s}^j(x)\cic{1}_{E_\nu}(x').
\]
Also, let $\psi$ be a smooth function on $\R$ with $\cic{1}_{[-2,2]}\leq\psi\leq\cic{1}_{[-4,4]} $. Let $\kappa\in [0,1)$ and define the multiplier operator
\[
\widehat{P_\nu^j} (\xi) =\psi(2^{j(1-\kappa)}\xi'\cdot e_\nu).
\]
We now  decompose
\[
H_{s}^j  :=   \Gamma_{s}^j + \Upsilon_s^j,   \qquad \Gamma^j_s:=\sum_{\nu} P_\nu^j* H_{s\nu}^j, \quad \Upsilon_s^j:=H_{s}^j -  \Gamma_{s}^j\]
 so that $\mathsf{H}^j$ is the sum of the single scale bilinear forms  
\begin{align*}
& \mathsf{G}_j(b,h)= \left\langle \sum_{s}   \Gamma_{s}^j* b_{s-j}, \overline h \right\rangle, \\
& \mathsf{U}_j(b,h)=\left\langle \sum_{s}   \Upsilon_{s}^j *b_{s-j},\overline h \right\rangle
\end{align*}
satisfying  the estimates below.
\begin{lemma} \label{Glemma} Let $\tau>1$.
Then
\[
|\mathsf{G}_j(b,h)| \leq C_\tau 2^{-j\frac{(1-\kappa)}{2}} \|\Omega_j\|_{\infty}  |Q|\|b\|_{{\mathcal X_1}}\|h\|_{\mathcal Y_\tau}, \qquad C_\tau= \frac{C\tau}{\tau-1}.\]
\end{lemma}
\begin{lemma} \label{Ulemma}Let $b\in \dot{\mathcal X}_1$. For all $\eps>0$ there exists a constant $C_{\kappa,\eps}$ depending on $\kappa,\eps$ only such that
\[
|\mathsf{U}_j(b,h)| \leq C_{\kappa,\eps}  2^{- \eps j} \|\Omega_j\|_{\infty} |Q| \|b\|_{\dot{\mathcal X}_1}\|h\|_{\mathcal Y_\infty}.\]
\end{lemma}
Notice that the combination of Lemma \ref{Glemma} with $\tau=2$ and $\kappa=1/2$ and Lemma \ref{Ulemma} with $\eps=1/4$ yields the required  estimate for  Lemma \ref{lemmaHinfty}, with $c=1/4$.
 Lemma \ref{lemmaHinfty} is thus proved up to the arguments for  Lemmata \ref{Glemma} and \ref{Ulemma}.
\begin{proof}[Proof of Lemma \ref{Glemma}] We may factor out $\|\Omega_j\|_\infty$ and assume that the angular part in the definition of $\Gamma_j$ is bounded by 1. We can also assume that $H_{s\nu}^j$ and $b$ are positive as cancellation plays no role in this argument: this is just a matter of saving space in the notation. Using interpolation and duality with $t$ below being the dual exponent of $\tau$, the estimate of the lemma follows if we show that for each integer $r\geq 1$ and $t=2r$  
\begin{equation}
\label{Glemma1}
\frac{1}{|Q|^{\frac1t}}\Big\|\sum_{s}   \Gamma_{s}^j* b_{s-j} \Big\|_{{t}}\lesssim t   2^{- \frac{j(1-\kappa)}{2}} 
  \|b\|_{\mathcal X_1} 
\end{equation}  with an implicit  constant that does not depend on $r$.
Setting
\[
M_\nu= \sum_{s} P^j_\nu * H_{s\nu}^j* b_{j-s}, \qquad D_\nu =\sum_{s}  H_{s\nu}^j* b_{s-j},
\]
we rewrite the left hand side of \eqref{Glemma1} raised to $t$-th power and subsequently estimate
\begin{equation}
\label{Glemma2}
\begin{split}
\Big\|\sum_{\nu_1,\ldots,\nu_{r}} \prod_{k=1}^{r}    M_{\nu_k} \Big\|_2^2& = 
 \Big\|\sum_{\nu_1,\ldots,\nu_{r}}    \widehat{M_{\nu_1}}*\cdots *  \widehat{M_{\nu_{r}}}  \Big\|_2^2
\lesssim 2^{rj(d-2+\kappa)}\sum_{\nu_1,\ldots,\nu_{r}}    \Big\| \prod_{k=1}^{r}    D_{\nu_k}   \Big\|_2^2  \\& \lesssim 2^{{t}j(d-1)} 2^{-rj(1-\kappa)} \sup_{\nu} \|D_\nu\|_{{t}}^{t}.
\end{split}
\end{equation}
We have   used Plancherel for the first equality, followed by the observation that  $\widehat P^j_{\nu_k} (\xi) $ is  uniformly bounded  and nonzero only if $|\xi'-e_{\nu_k}|<2^{-j(1-\kappa)}$. Thus there are at most $C2^{rj(d-2+\kappa)}$ $r$-tuples such that the $r$-fold convolution is nonzero, whence the first bound. Another usage of Plancherel,  the observation  that there are at most $2^{rj(d-1)}$ tuples in the summation, and finally H\"older's inequality yield the  second bound.
We are thus done if we   estimate for each fixed $\nu$
\begin{equation}
\label{dnu}
  \sum_{s_{1}\geq \cdots\geq s_{{t}}} \int \left(\prod_{k=1}^{{t}} H_{s_k\nu}^j(x-y_k) b_{s_k-j}(y_k) \right) \d y_1\cdots \d y_{t} \d x  \lesssim C^t  2^{-tj(d-1)}|Q|\|b\|_{\mathcal X_1}^{t} 
\end{equation}
as $\|D_\nu\|_{t}^{t}$ is at most $t^{{t}}$ times the above integral.
Notice that if $ \sigma\leq s$ then $\supp H_{\sigma\nu}^j$ is contained in a box  $R_{s}$ centered at zero and  having one long side of length $\lesssim 2^{s}$ and $(d-1)$ short sides of length $2^{s-j}$. 
If $z\in \R^d$,  $R_s(z)=z+R_s$ and
\[
\mathcal Q_s(z) = \{L\in \mathcal Q: s_L\leq s-j, L\subset 100 R_s(z)\}, \qquad \mathsf{b}_{R_s(z)} := \sum_{L \in \mathcal Q_s(z)} b_L
\]
 we have by disjointness of $ L\in \mathcal Q $
\begin{equation}
\label{maintrick}
2^{-sd}\left\|\mathsf{b}_{R_s(z)}\right\|_1 \lesssim  2^{-sd} |R_{s}(z)|  \|b\|_{\mathcal X_1}\leq   C 2^{-j(d-1)} \|b\|_{\mathcal X_1}:= \alpha.
\end{equation}
Also notice that for all fixed $y_1,\ldots,y_{{t}}$ and for all $s_1\geq \cdots \geq s_{t}$ there holds
\[I_{s_1,\ldots,s_{{t}}}(y_1,\ldots, y_{t}):=\int
\left(\prod_{k=1}^{{t}} H_{s_k\nu}^j(x-y_k)   \right)\,   \d x \leq \|H_{s_{t}\nu}^j\|_1 \prod_{k=1}^{{t}-1} \|H_{s_k\nu}^j\|_\infty \lesssim 2^{-j(d-1)}   2^{- d \mathsf{s}_{{t}-1} }
\]
where we wrote, here and in what follows
\[
\mathsf{s}_{n} = \sum_{k=1}^n s_k, \qquad n=1,\ldots, t.
\]
Furthermore, 
 $I_{s_1,\ldots,s_{{t}}}(y_1,\ldots, y_{t})$ is nonzero only if $y_{k}\in 2R_{s_{k-1}}(y_{k-1})$ for $k=t,t-1,\ldots,2$.
Now, writing $\mathsf{b}_{s_k}$ in place of ${b}_{s_k-j}$  for reasons of space as $j$ is kept fixed throughout and using  \eqref{maintrick} repeatedly,    the sum in \eqref{dnu} is equal to 
\[
\begin{split}
 &\quad\sum_{s_{1}\geq \cdots\geq s_{{t}}}  \int I_{s_1,\ldots,s_{{t}}}(y_1,\ldots ,y_{t}) \left( \prod_{k=1}^{{t}}   \mathsf{b}_{s_k}(y_k) \right)\, \d y_1\cdots \d y_{t}   \\ &\lesssim 2^{-j(d-1)} \sum_{s_1\geq\cdots \geq s_{{t}-1} } 2^{-d\mathsf{s}_{{t}-2} } \int  \mathsf{b}_{s_1}(y_1)\left( \prod_{k=2}^{{t}-1}   \mathsf{b}_{s_{k}}(y_k)  \cic{1}_{2R_{s_{k-1}}(y_{k-1})}(y_k)  \right)
 \frac{\|\mathsf{b}_{R_{s_{{t}-1}}(y_{{t}-1})}\|_1}{{2^{d  s_{{t}-1}}}} \d y_1\cdots \d y_{{t}-1}
 \\ & \lesssim \alpha  2^{-j(d-1)} \sum_{s_1\geq\cdots \geq s_{{t}-2} } 2^{-d\mathsf{s}_{{t}-3}} \int  \mathsf{b}_{s_1}(y_1)\left( \prod_{k=2}^{{t}-2}   \mathsf{b}_{s_{k}}(y_k)  \cic{1}_{2R_{s_{k-1}}(y_{k-1})}(y_k)  \right)
 \frac{\|\mathsf{b}_{R_{s_{{t}-2}}(y_{{t}-2})}\|_1}{{2^{d  s_{{t}-2}}}} \d y_1\cdots \d y_{{t}-2} 
  \\ & \lesssim \cdots \lesssim \alpha^{{t}-1} 2^{-j(d-1)} |Q|\|b\|_{\mathcal X_1} \leq C^{t}  2^{-tj(d-1)}|Q|\|b\|_{\mathcal X_1}^{t}\end{split}
\]
  as claimed, and this completes the proof.
  \end{proof}   
\begin{proof}[Proof of Lemma \ref{Ulemma}]
Again we  factor out $\|\Omega_j\|_\infty$ and work under the assumption that the angular part  is bounded by 1.
 In this proof $M$ is a large integer whose value may differ at each occurrence and the constants implied by the almost inequality sign are allowed to depend on $M$ only.    Let  $\beta$ be a smooth  function supported in $A_1=\{2^{-1}\leq |\xi|\leq 2\}$ and satisfying  
\[
\sum_{k\in \mathbb Z} \beta^2 (2^{k} \xi) =1 \qquad \xi \neq 0.
\] 
 Denote by
$
B_k =\mathcal F^{-1}\big\{ \beta(2^k\cdot)\big\}.
$
 Defining 
 \[
 \widehat{R_{s\nu}^{jk}}(\xi) = \beta(2^{k}\xi) \left( 1- \widehat{P_{\nu}^j}(\xi)\right) \widehat{H_{s\nu}^j}(\xi),
 \]
we recall from \cite[eqs.\ (2.6), (2.7)]{Seeger} the estimate
\[
\left\|R_{s\nu}^{jk} \right\|_{1} \lesssim_M 2^{-j(d-1)} \min\left\{1, 2^{-M\kappa j}2^{-M(s-j-k)}\right\}. 
\]
Now, fix $s$ and $L\in \mathcal Q$ with   $\ell(L)=2^{s-j}$ for the moment.  Recalling the definition of $\Upsilon_{s}^j$, we have the decomposition
 \[
 |\l \Upsilon_s^j*b_L, \overline h\r| \leq   \sum_{\nu}\sum_{k} | \l R_{s\nu}^{jk} *B_k * b_{L}, \overline h\r|, \]
and the 
  cancellation estimate (cf.\ \cite[eq.\ (2.5)]{Seeger}, a simpler version of Lemma \ref{bbestimate1})
\begin{equation}
\label{seegerest}
\begin{split}  
&\quad | \l R_{s\nu}^{jk} *B_k * b_{L}, \overline h\r|     \lesssim 
\min\{1,2^{(s-j)-k}\} 
\|R_{s\nu}^{jk} \|_1 \| b_{L}\|_1 \|h\|_\infty 
\\ &  \lesssim   2^{-j(d-1)} \min\left\{2^{(s-j)-k}, 2^{-M\kappa j-M(s-j-k)}\right\}|L|\|b\|_{\dot{\mathcal X}_1}  \|h\|_{\mathcal Y_\infty}. \end{split} 
\end{equation} 
Note  that $\#\Xi\lesssim 2^{j(d-1)}$. So for each $\eps>0$ we can use the left estimate in \eqref{seegerest}  for $k\geq s-j(1-\eps)$ and the right estimate otherwise, and obtain   
\begin{equation}
\label{intermest}
|\l \Upsilon_s^j*b_L, \overline h\r| \leq   \sum_{\nu}\sum_{k} | \l R_{s\nu}^{jk} *B_k * b_{L}, \overline h\r| \lesssim 2^{- \eps j}|L|\|b\|_{\dot{\mathcal X}_1}  \|h\|_{\mathcal Y_\infty}
\end{equation}
provided that $M$ is chosen large enough to have $2\eps<M\kappa$. The proof is thus completed by summing \eqref{intermest} over  $L \in \mathcal Q$ with $\ell(L)=2^{s-j}$ and later over $s$.\end{proof}

\section{Proof of Theorem \ref{theoremBR}}\label{secBR} Throughout this proof, $C$ is a positive absolute dimensional constant which may vary at each occurrence without explicit mention. Most of the arguments in this Section are contained in \cite[Section 3]{Ch88}; we reproduce the details for clarity.

Let 
$
\psi(x) = \cos\left( \textstyle 2\pi(|x| -\delta/4)\right).
$
From the asymptotic expansion of the inverse Fourier transform of the multiplier of $B_\delta$ \cite[Section 3]{Ch88}, which is $\mathcal C^\infty$ and radial, we obtain the kernel representation
\[
B_\delta (x) =\sum_{s\geq 1} \sum_{\nu} K_{s,\nu}(x) + L(x),  
\]
Here
\[
K_{s,\nu}(x)=\Omega_\nu(x') \psi(x) 2^{-sd}\phi(2^{-s} x)\]
with $\Omega_\nu$ being a finite smooth partition of unity on the unit sphere $S^{d-1}$ with sufficiently small support which is introduced for technical reasons, 
and $\phi$ being a suitable smooth radial function supported in $A_1=\{2^{-2} \leq |x|\leq 1\}$, while $L(x)$ is an integrable kernel  with $L(x)\leq C(1+|x|)^{-(d+1)}$, so that
\[
Lf(x) \leq C \mathrm{M}_1 f(x)
\]
which can be ignored for our purposes. We can also think of $\nu$ as fixed and omit it from the notation, and consider the kernel $K=\{K_s\}$ as above. We are going to verify that conditions in Theorem \ref{theoremABS} are satisfied by (the dual form to) $B_\delta$. First of all, condition \eqref{assker2} is obvious from the above discussion as $[K]_{0,\infty}<\infty$. Second, the \eqref{asstrunc} condition follows from the well-known estimate 
\[
\sup_{\mu,\nu}\|\Lambda_{\mu}^\nu\|_{L^2(\R^d)\times L^2(\R^d)}\leq C,
\]
see for instance \cite[Theorem E]{DR}. In order to verify the condition \eqref{lest}, let $\mathcal Q$ be a stopping collection with top $Q$. Let $b\in \mathcal X_1(\mathcal Q)$; we change a bit the notation for $b_s$ in this context by redefining
\[
b_s:=\sum_{ s_L=  s} b_L, \quad  s\geq 1, \qquad b_0:=\sum_{ { s_L\leq  0}} b_L.
\]
It is easy to see that in this context if $b\in \mathcal X_1$   supported on $Q$ and $h \in \mathcal Y_1$ one has
\[
\Lambda_{\mathcal Q,\mu,\nu} (b,h)= \left\langle \sum_{j\geq 1} \sum_{s\geq j}  \eps_s K_{s}*  b_{s-j},\overline h\right\rangle
\]
for a suitable choice of  signs  $\{\eps_s\}\in \{-1,0,1\}^{\mathbb Z}$,
and the same for $\Lambda_{\mathcal Q,\mu,\nu} (h,b)$ up to replacing $b$ by $b^{\mathsf{in}}$, restricting $h$ to be supported on $Q$, transposing $K_s$,  and  subtracting off the remainder terms which are estimated by 
\[|Q| \|b\|_{{\mathcal X_{1}} }  \|h\|_{\mathcal Y_{1}}.
\]  
Theorem \ref{theoremBR} is thus obtained from the next proposition via  an application of Theorem \ref{theoremABS}.\begin{proposition}  \label{BRest}
 Let $\{\eps_s\}\in \{-1,0,1\}^{\mathbb Z}$ be a choice of signs, $b \in {\mathcal X}_1$ and define
\[
 \mathsf{K}(b,h):=\left\langle \sum_{j\geq 1} \sum_{s\geq j}   \eps_s K_{s}*  b_{s-j}, \overline h\right\rangle.
\]
 There exists an absolute  constant $C$, in particular uniform over   $\{\eps_s\}\in \{-1,0,1\}^{\mathbb Z}$,  such that  
 \[
\left| \mathsf{K}(b,h)\right| \leq   \frac{Cp}{p-1}  |Q| \|b\|_{{\mathcal X_{1}} }  \|h\|_{\mathcal Y_{p}}.
\]
\end{proposition}
Notice that here we do not need to require $b\in \dot{ \mathcal X}_1$ as per the oscillatory nature of the problem.
\subsection{Proof of Proposition \ref{BRest}} Given our choice of  $\{\eps_s\}\in \{-1,0,1\}^{\mathbb Z}$, we  relabel $K_s:=\eps_sK_s$. It will be clear from the proof that the signs $\eps_s$ play no role.
We split
\[
\mathsf{K}(b,h)= \sum_{j\geq 1}\mathsf{K}^j(b,h),\qquad \mathsf{K}^j(b,h):= \sum_{s\geq j}  \left\langle K_{s}*  b_{s-j}, \overline h\right\rangle.
\]
The first estimate is a trivial one.
\begin{lemma} There exists $C>0$ such that $|\mathsf{K}^j(b,h)|\leq C|Q|\|b\|_{\mathcal X_1}\|h\|_{\mathcal Y_1}.$
\end{lemma} 
\begin{proof}   This follows from applying Lemma \ref{trivialestlemma} with $\beta=\infty $ to $K=\{K_s\}$, as it is immediate to see that for this kernel one has $[K]_{0,\infty}\leq C$ as already remarked. \end{proof}
The second estimate, which is essentially contained in \cite[Section 3]{Ch88}, is the one providing decay. 
\begin{lemma} \label{decayBR} There exists $C,c>0$ such that $|\mathsf{K}^j(b,h)|\leq C2^{-cj}|Q|\|b\|_{\mathcal X_1}\|h\|_{\mathcal Y_2}.$
\end{lemma} 
It is easy to see that interpolating the above estimates yields 
$$|\mathsf{K}^j(b,h)|\leq C2^{-j\frac{c(p-1)}{p}}|Q|\|b\|_{\mathcal X_1}\|h\|_{\mathcal Y_p},$$
the summation of which yields Proposition \ref{BRest}.
\begin{proof}[Proof of Lemma \ref{decayBR}]    Let $\widetilde K_s(\cdot)=\overline{K_s(-\cdot)}.$
We recall from \cite[Lemma 3.1]{Ch88} the estimates
\begin{equation}
\label{BRlemma}
\begin{split}
& |K_s*\widetilde K_s(x)| \leq C 2^{-ds} (1+|x|)^{-\delta},
\\
& \|K_s*\widetilde K_t \|_\infty \leq C 2^{-dt} 2^{-\delta s},\qquad \forall s<t-1.
\end{split}
\end{equation}
By duality, it suffices to prove that
\begin{equation}
\label{mainBR} 
\left\| K_{j}*  b_{0}\right\|_{2}^2
+
\left\|  \sum_{s>j} K_{s}*  b_{s-j}\right\|_{2}^2   \leq C2^{-cj}|Q|  \|b\|_{\mathcal X_1}^2.
\end{equation}
For the first term we use the first estimate in \eqref{BRlemma}:
\[
\left\| K_{j}*  b_{0}\right\|_{2}^2= |\l b_{0}, K_{j}* \widetilde K_{j}* b_0  \r | \leq \|b_0\|_1\|K_{j}* \widetilde K_{j}*b_0\|_\infty\leq C 2^{-\min(\delta,d)j}|Q|  \|b\|_{\mathcal X_1}^2.
\]The last inequality above follows from
\[
\|K_{j}* \widetilde K_{j}* b_0\|_{\infty}\leq 2^{-jd}\sum_{m=0}^j 2^{-m\delta}\sup_{x\in\mathbb{R}^d}\|b_0\|_{L^1(B(x,C2^m))}\leq C 2^{-\min(\delta,d)j}\|b\|_{\mathcal{X}_1},
\]where $B(x,C2^m)$ denotes a ball centered at $x$ with radius $C2^m$.
For the second term, we begin by quoting from \cite[(3.2)]{Ch88} that
\begin{equation}
\label{BRlemma2}
\left\| K_{s}*  b_{s-j}\right\|_{2}^2 \leq C 2^{-\delta j} \|b\|_{\mathcal X_1} \|b_{s-j}\|_1.
\end{equation}
Observe that
\begin{equation}
\label{BRlemma3}
\begin{split}
 \left\| \sum_{s>j} K_{s}*  b_{s-j}\right\|_{2}^2   & \leq   \sum_{s>j}  \left\|K_{s}*  b_{s-j}\right\|_{2}^2 + 2 \sum_{s} \left|\langle K_{s}*  b_{s-j}, K_{s-1}*  b_{s-1 -j} \rangle  \right|  \\ & +  2 \sum_{t} \sum_{j<s<t-1} \left|\langle \widetilde{K}_t*K_{s}*  b_{s-j}, b_{t- j} \rangle  \right|. 
\end{split}
\end{equation}
The first two terms are bounded by 
\[
C 2^{-\delta j} \|b\|_{\mathcal X_1} \sum_{s}\|b_{s-j}\|_1 \leq C 2^{-\delta j} |Q|\|b\|_{\mathcal X_1}^2,  \]
 according to \eqref{BRlemma2} for the first one and Cauchy-Schwarz followed by \eqref{BRlemma2} for the second. For the third term, from the second estimate of \eqref{BRlemma} and support considerations one has
\[
\|\widetilde{K}_t*K_{s}*  b_{s-j}\|_\infty\leq C \left(\sup_{x\in \R^d} \|b_{s-j}\|_{L^1(B(x,C2^t))} \right) \|\widetilde{K}_t*K_{s} \|_\infty\leq   C2^{-\delta s}\|b\|_{\mathcal X_1}.   
\] Therefore, the third summand in \eqref{BRlemma3} is dominated by
\[
 C\|b\|_{\mathcal X_1}\sum_{t>j} \|b_{t-j}\|_1\sum_{j<s<t-1} 2^{-\delta s}\leq C 2^{-\delta j} |Q|\|b\|_{\mathcal X_1}^2
\]
and collecting all the above estimates \eqref{mainBR} follows.  \end{proof}

%

\appendix
\section{Verification of \eqref{Kokadj}-\eqref{GNI}} \label{ssadj}
 Let $\mathcal Q$ be a stopping collection with top $Q$, $h\in \mathcal{Y}_{q'}, b\in\mathcal{X}_{q'}$. Clearly we can assume $\supp h\subset Q$. By possibly replacing $K_s$ by zero when $s\not\in (\mu,\nu]$ we can ignore the truncations $\mu,\nu$ in what follows and omit them from the notation.
Recall  the definitions \eqref{therealLambdaq}, \eqref{therealLambdaQ}
 \[
 \Lambda_{\mathcal Q}(h,b) =  \Lambda_{Q}(h,b) - \sum_{{\substack{R\in \mathcal Q\\ R\subset Q}}} \Lambda_{R}(h,b)  =  \Lambda^{ s_{Q} }(h ,b)- \sum_{{\substack{R\in \mathcal Q\\ R\subset Q}}} \Lambda^{ s_{R} }(h \cic{1}_{R},b)  \]
and the decomposition
\[b=b^{\mathsf{in}}+ b^{\mathsf{out}}, \qquad 
b^{\mathsf{in}} = \sum_{\substack{L\in \mathcal Q\\ 3L \cap 2Q \neq \emptyset}} b_L, \qquad b^{\mathsf{out}} = \sum_{\substack{L\in \mathcal Q\\ 3L \cap 2Q = \emptyset}} b_L.
\]
We first estimate
\begin{equation} \label{bout}
|\Lambda_{\mathcal Q}(h,b^{\mathsf{out}})| \lesssim  [K]_{0,q} |Q| \|h\|_{{\mathcal Y}_{1} }   \|b\|_{{\mathcal X}_{q'} }
\end{equation}
which is   a single scale estimate. In fact, since $\dist(R,\supp b^{\mathsf{out}})\geq\ell(R)/2$ for all   $R\subset Q$, by virtue of the support restriction in \eqref{assker2}, 
\[
s< s_R\implies \int K_s(x,y) h(y) \cic{1}_{R}(y) b^{\mathsf{out}}(x) \, \d y \d x =0.
\]
Therefore, by the same argument used in \eqref{tailsest},
\begin{equation} \label{bouttail}
\begin{split} &\quad 
|  \Lambda^{ s_{Q} }(h ,b^{\mathsf{out}})|\leq \int |K_{s_Q}(x,y)| |h (y)| |b^{\mathsf{out}}(x)| \, \d y \d x   \lesssim  [K]_{0,q} |Q|\|h\|_{{\mathcal Y}_{1} }  \|b\|_{{\mathcal X}_{q'} }.
\end{split}
\end{equation}
Proceeding similarly,  if $R\in \mathcal Q, R\subset Q$
\[
\begin{split}  
|  \Lambda^{ s_{R} }(h \cic{1}_R ,b^{\mathsf{out}})|
\leq \int |K_{s_R}(x,y)| |h\cic{1}_R (y)| |b^{\mathsf{out}}(x)| \, \d y \d x   \lesssim  [K]_{0,q} |R|\|h\|_{{\mathcal Y}_{1} }  \|b\|_{{\mathcal X}_{q'} }. 
\end{split}
\]
and the claimed \eqref{bout} follows by summing the last display over $R\in \mathcal Q, R\subset Q$, which are pairwise disjoint, and combining the result with \eqref{bouttail}.
The representation \eqref{Kokadj} will then be a simple consequence of the equality
\begin{equation} \label{final2}
\Lambda_{\mathcal Q}(h,b^{\mathsf{in}})= \left( \Lambda^{s_{Q}}(h,b^{\mathsf{in}}) - \sum_{L\in \mathcal Q: 3L\cap 2Q\neq \emptyset}  \Lambda^{s_{L}}(h,b_L)\right)   + V_{ \mathcal Q }(h, b) 
\end{equation}
where the remainder $V_{\mathcal Q}$ satisfies
\begin{equation} \label{GNIapp}  \begin{split}
|V_{ \mathcal Q }(h, b)|& \lesssim [K]_{0,q} |Q| \|h\|_{{\mathcal Y}_{q'} }   \|b\|_{{\mathcal X}_{q'} }.\end{split}
\end{equation}
We turn to the proof of \eqref{final2}. We will   use below without explicit mention that whenever $L,R\in \mathcal Q$ with $3R\cap3 L \neq \emptyset$, then $|s_{L}-s_R|<8,$ a consequence of the separation property \eqref{separation}. First of all, the restriction on the support \eqref{assker2} entails that
\begin{equation}
\label{appeq1}
\sum_{R\in \mathcal Q}  \Lambda^{s_{R}}(h\cic{1}_R,b^{\mathsf{in}}) =  \sum_{ {R\in \mathcal Q }}  \sum_{\substack {L\in \mathcal Q \\ 3L\cap 3R \neq \emptyset\\ 3L\cap 2Q \neq \emptyset}}\Lambda^{s_{R}}(h\cic{1}_R,b_L)   \end{equation}
as $\Lambda^{s_{R}}(h\cic{1}_R,b_L)=0$ unless $3L\cap 3R$ is nonempty.
As there are at most 16 $s$-scales in each difference $\Lambda^{s_{L}}-\Lambda^{s_{R}}$, using the trivial estimate \eqref{tailsest} with $\beta=q$ for each such scale yields
\begin{equation}
\label{appeq11}
\begin{split}
& \quad \sum_{R\in \mathcal Q }\sum_{\substack {L\in \mathcal Q \\ 3L\cap 3R \neq \emptyset\\ 3L\cap 2Q \neq \emptyset}}
\left|\Lambda^{s_{L}}(h\cic{1}_R,b_L) -\Lambda^{s_{R}}(h\cic{1}_R,b_L) \right| \lesssim  [K]_{0,q} \|h\|_{\mathcal Y_{q'}} \sum_{R\in \mathcal Q }  \sum_{\substack {L\in \mathcal Q \\ 3L\cap 3R \neq \emptyset\\ 3L\cap 2Q \neq \emptyset}}\|b_L\|_{1}\\ & \lesssim [K]_{0,q} \|h\|_{\mathcal Y_{q'}}\|b\|_{\mathcal X_1} \sum_{R\in \mathcal Q } |R| \lesssim [K]_{0,q}   |Q|\|h\|_{\mathcal Y_{q'}}\|b\|_{\mathcal X_1}.
\end{split}\end{equation}
Recalling the second property of stopping collections in \eqref{separation},  we have  the decomposition  
 \[ 
 h = h^{\mathsf{in}} + h^{\mathsf{out}}, \qquad h^{\mathsf{in}}:= h \cic{1}_{ \bigcup_{R\in \mathcal Q} R}, \qquad \supp h^{\mathsf{out}} \cap \left( \bigcup_{L\in\mathcal Q: 3L\cap 2 Q\neq \emptyset} 9L \right)=\emptyset.
 \]
Therefore, up to including the  error term of \eqref{appeq11} in \eqref{GNIapp},    \eqref{appeq1} can be rewritten as  
\begin{equation} \label{appeq2}
\begin{split} & \sum_{R\in \mathcal Q}  \sum_{\substack {L\in \mathcal Q \\ 3L\cap 3R \neq \emptyset\\ 3L\cap 2Q \neq \emptyset}}\Lambda^{s_{L}}(h\cic{1}_R,b_L)  
  =   \sum_{\substack {L\in \mathcal Q \\ 3L\cap 2Q \neq \emptyset} } \Lambda^{s_{L}}(h^{\mathsf{in}},b_L) - \sum_{\substack {L\in \mathcal Q \\ 3L\cap 2Q \neq \emptyset} } \Lambda^{s_{L}}(\widetilde{h_L},b_L),
\\
&\widetilde{h_L} =  \sum_{\substack {R\in \mathcal Q \\ 3L\cap 3R =\emptyset}} h\cic{1}_R, \qquad \supp \widetilde{h_L}\subset \R^d \setminus 3L.
  \end{split}
\end{equation}
We note that all the terms in the second sum on the right hand side of the first line of \eqref{appeq2} vanish due to the support restriction on $K_s$, as all the scales appearing are less than or equal to $s_L$ and $\supp b_L\subset L$.   The reasoning beginning with decomposition \eqref{appeq1} leads thus to the equality, up to tolerable error terms
\begin{equation} \label{appeq3}
\begin{split}& \sum_{R\in \mathcal Q}  \Lambda^{s_{R}}(h\cic{1}_R,b^{\mathsf{in}}) = \sum_{\substack {L\in \mathcal Q \\ 3L\cap 2Q \neq \emptyset} } \Lambda^{s_{L}}(h,b_L) - \sum_{\substack {L\in \mathcal Q \\ 3L\cap 2Q \neq \emptyset} } \Lambda^{s_{L}}(h^{\mathsf{out}},b_L).
  \end{split}
\end{equation}
Finally the second term on the right hand side of \eqref{appeq3} also vanishes, by virtue of the restriction on the support of $h^{\mathsf{out}}$, which does not intersect $9L$ for any $L$ in the sum. Therefore, \eqref{appeq3} is actually the equality  
\[
\sum_{\substack{R\in \mathcal Q\\ R\subset Q}} \Lambda^{s_R}(h\cic{1}_{R}, b^{\mathsf{in}})= \sum_{ R\in \mathcal Q } \Lambda^{s_R}(h\cic{1}_{R}, b^{\mathsf{in}})= \sum_{\substack {L\in \mathcal Q \\ 3L\cap 2Q \neq \emptyset} } \Lambda^{s_{L}}(h,b_L) +   V_{ \mathcal Q }(h, b)
\]
where $V_{ \mathcal Q }(h, b)$ satisfies \eqref{GNIapp}; the first equality in the above display is due to $\supp h\subset Q$. This equality clearly implies the sought after \eqref{final2}. 

\section{Sparse domination implies weak $L^1$ estimate}\label{Weak11}
We show that if a sublinear operator $T$ satisfies the sparse estimate \eqref{sparsegen} for $p_1=1,p_2=r$ for some $1\leq r<\infty$ then $T$ is of  weak type $(1,1)$. In particular, as mentioned in the Introduction, together with Theorem \ref{theoremRH}, this yields the weak $L^1$ estimate of $T_\Omega$, which is the main result of \cite{Seeger} proved by Seeger. 
The proof that follows is a simplified version of the arguments in \cite[Appendix A]{CuDPOu}; we are sure these arguments are well-known but were unable to locate a precise reference.
\begin{theorem} \label{domshift} Suppose that the sublinear operator $T $ has the following property: there exists $C>0$ and $1\leq r<\infty$ such that for every $f_1,f_2$ bounded with compact support   there exists a sparse collection $\mathcal S$ such that\begin{equation}
\label{unifest}
|\l Tf_1,f_2\r| \leq C   \sum_{Q \in \mathcal S } |Q| \l f_1 \r_{1,Q} \l f_2 \r_{r,Q} . 
\end{equation}
Then $T:L^{1}(\R^d) \to L^{1,\infty}(\R^d)$ boundedly.
\end{theorem}
\begin{proof}
By standard arguments it suffices to verify that
\[
\sup_{\|f_1\|_1=1}  \sup_{G\subset \R^d} \inf_{\substack{G'\subset G\\ |G|\leq 2|G'|}}\sup_{|f_2|\leq \cic{1}_{G'}} |\l Tf_1,f_2 \r|\leq C
\]
where $f_1,f_2$ are bounded  and compactly supported and $G$ has finite measure.
Given such $f_1$ with $\|f_1\|_1=1$ and $G$ of finite measure define the sets
\[\begin{split}
&  H :=  \left\{ x \in \R^d:  \mathrm{M}_{1} f_{1}(x)   > C |G|^{-1} \right\}, \\ &\tilde H:=\bigcup_{Q \in \mathcal Q} 3Q, \qquad \mathcal Q=\left\{ \textrm{max.\ dyad.\ cube\ } Q : |Q \cap   H|\geq 2^{-5} |Q|\right\}.
\end{split}
\]
It is  easy to see that $|\tilde H|\leq 2^{-10}|G|$ for suitable choice of $C$. 
  Therefore the set $G':G\backslash  \tilde H$ satisfies $|G|\leq 2|G'|$. We make the preliminary observation that
  \[
 \sup_{x\in H^c} \mathrm M_1f_1(x) \leq C|G|^{-1},
  \]
 so that by interpolation
 \begin{equation}
\label{2}
\|\mathrm M_1 f_1\|_{L^{p'}( H^c)} \leq  \left(\sup_{x\in H^c} \mathrm M_1f_1(x) \right)^{1-\frac{1}{p'}} \|\mathrm M_1 f_1\|_{1,\infty}^{\frac{1}{p'}} \leq C|G|^{-(1-\frac{1}{p'})},
\end{equation}
where $p'>1$ is chosen such that $p>r$. Fixing now any $f_2$ restricted to $G'$, we apply the domination  estimate, yielding the existence of a sparse collection $\mathcal S$ for which we have the estimate
\[
|\l Tf_1,f_2\r| \leq C \sum_{Q \in \mathcal S } |Q| \l f_1 \r_{1,Q} \l f_2 \r_{r,Q} . \]
We claim that 
\begin{equation}
\label{Isk}
|Q\cap H| \leq 2^{-5} |Q| \qquad \forall Q \in \mathcal S. 
\end{equation}
This is because if \eqref{Isk} fails for $Q$, $Q$ must be contained in $3Q'$ for some $Q'\in \mathcal Q$. But the support of $f_{2}$ is contained in $\widetilde{H}^c$ which does not intersect $3Q'$, whence  $\l f_{2}\r_{r,Q}=0$.
Relation \eqref{Isk} has the consequence that  if $\{E_Q: Q\in \mathcal S\}$ denote the distinguished pairwise disjoint  subsets of $Q \in \mathcal S$ with $|E_Q| \geq 2^{-2} |Q|$, the sets  $\widetilde{E_Q}:= E_Q \cap H^c$ are also pairwise disjoint and  $|\widetilde{E_Q}| \geq 2^{-3} |Q|$.   Therefore,
since 
the union of   $\widetilde{E_Q}$ is contained in $H^c$ by standard arguments we arrive at
\[
\begin{split}
|\l Tf_1,f_2\r| & \leq C \sum_{Q \in \mathcal S } |Q| \l f_1 \r_{1,Q} \l f_2 \r_{r,Q} \\
&\leq C \sum_{Q \in \mathcal S } |\widetilde{E_Q}| \l f_1 \r_{1,Q} \l f_2 \r_{r,Q} \leq C \int_{H^c} \mathrm M_1 f(x) \mathrm M_r f_2(x) \, \d x\\ &\leq C \|\mathrm M_1 f_1\|_{L^{p'}( H^c)} \|\mathrm M_r f_2\|_{L^p(\mathbb{R}^d)} \leq C|G|^{-(1-\frac{1}{p'})}|G|^{\frac1p} \leq
C \end{split}
\]
using \eqref{2} in the last step. The proof is complete.
\end{proof}

\bibliography{RoughKernels}{}

\providecommand{\bysame}{\leavevmode\hbox to3em{\hrulefill}\thinspace}
\providecommand{\MR}{\relax\ifhmode\unskip\space\fi MR }
\providecommand{\MRhref}[2]{%
  \href{http://www.ams.org/mathscinet-getitem?mr=#1}{#2}
}
\providecommand{\href}[2]{#2}
\begin{thebibliography}{10}

\bibitem{BBP}
Cristina Benea, Fr\'ed\'eric Bernicot, and Teresa Luque, \emph{Sparse bilinear
  forms for {B}ochner {R}iesz multipliers and applications}, Trans.\ London
  Math. Soc. \textbf{4} (2017), no.~1, 110--128.

\bibitem{BFP}
Fr\'ed\'eric Bernicot, Dorothee Frey, and Stefanie Petermichl, \emph{Sharp
  weighted norm estimates beyond {C}alder{\'o}n-{Z}ygmund theory}, Anal. PDE
  \textbf{9} (2016), no.~5, 1079--1113. \MR{3531367}

\bibitem{BCDHL}
The~Anh Bui, Jos\'e~M. Conde-Alonso, Xuan~Thinh Duong, and Mahdi Hormozi,
  \emph{A note on weighted bounds for singular operators with nonsmooth
  kernels}, Studia Math. \textbf{236} (2017), no.~3, 245--269. \MR{3600764}

\bibitem{CZ}
A.~P. Calder{\'o}n and A.~Zygmund, \emph{On singular integrals}, Amer. J. Math.
  \textbf{78} (1956), 289--309. \MR{0084633}

\bibitem{CDS}
Mar{\`\i}a Carro and Carlos Domingo-Salazar, \emph{Weighted weak-type {$(1,1)$}
  estimates for radial fourier multipliers via extrapolation theory}, preprint.
  to appear in J. Anal. Math.

\bibitem{Ch88}
Michael Christ, \emph{Weak type {$(1,1)$} bounds for rough operators}, Ann. of
  Math. (2) \textbf{128} (1988), no.~1, 19--42. \MR{951506}

\bibitem{CRub}
Michael Christ and Jos{\'e}~Luis Rubio~de Francia, \emph{Weak type {$(1,1)$}
  bounds for rough operators. {II}}, Invent. Math. \textbf{93} (1988), no.~1,
  225--237. \MR{943929}

\bibitem{CR}
Jos{\'e}~M. Conde-Alonso and Guillermo Rey, \emph{A pointwise estimate for
  positive dyadic shifts and some applications}, Math. Ann. \textbf{365}
  (2016), no.~3-4, 1111--1135. \MR{3521084}

\bibitem{CMP}
David~V. Cruz-Uribe, Jos{\'e}~Maria Martell, and Carlos P{\'e}rez,
  \emph{Weights, extrapolation and the theory of {R}ubio de {F}rancia},
  Operator Theory: Advances and Applications, vol. 215, Birkh\"auser/Springer
  Basel AG, Basel, 2011. \MR{2797562}

\bibitem{CuDPOu}
Amalia Culiuc, Francesco Di~Plinio, and Yumeng Ou, \emph{Domination of
  multilinear singular integrals by positive sparse forms}, preprint
  arXiv:1603.05317.

\bibitem{CuDPOu2}
Amalia Culiuc, Francesco Di~Plinio, and Yumeng Ou, \emph{Uniform sparse domination of singular integrals via dyadic
  shifts}, preprint arXiv:1610.01958, to appear in Math.\ Res.\ Lett.

\bibitem{DpLer2013}
Francesco Di~Plinio and Andrei~K. Lerner, \emph{On weighted norm inequalities
  for the {C}arleson and {W}alsh-{C}arleson operator}, J. Lond. Math. Soc. (2)
  \textbf{90} (2014), no.~3, 654--674. \MR{3291794}

\bibitem{Duo1993}
Javier Duoandikoetxea, \emph{Weighted norm inequalities for homogeneous
  singular integrals}, Trans. Amer. Math. Soc. \textbf{336} (1993), no.~2,
  869--880. \MR{1089418}

\bibitem{DR}
Javier Duoandikoetxea and Jos{\'e}~L. Rubio~de Francia, \emph{Maximal and
  singular integral operators via {F}ourier transform estimates}, Invent. Math.
  \textbf{84} (1986), no.~3, 541--561. \MR{837527}

\bibitem{GS2}
L.~Grafakos and A.~Stefanov, \emph{Convolution {C}alder\'on-{Z}ygmund singular
  integral operators with rough kernels}, Analysis of divergence ({O}rono,
  {ME}, 1997), Appl. Numer. Harmon. Anal., Birkh\"auser Boston, Boston, MA,
  1999, pp.~119--143. \MR{1731263}

\bibitem{HPR}
Tuomas Hyt{\"o}nen, Carlos P{\'e}rez, and Ezequiel Rela, \emph{Sharp reverse
  {H}\"older property for {$A_\infty$} weights on spaces of homogeneous type},
  J. Funct. Anal. \textbf{263} (2012), no.~12, 3883--3899. \MR{2990061}

\bibitem{HRT}
Tuomas~P. Hyt\"onen, Luz Roncal, and Olli Tapiola, \emph{Quantitative weighted
  estimates for rough homogeneous singular integrals}, Israel J. Math.
  \textbf{218} (2017), no.~1, 133--164. \MR{3625128}

\bibitem{KL}
Ben Krause and Michael~T. Lacey, \emph{Sparse bounds for random discrete
  {C}arleson theorems}, preprint arXiv:1609.08701.

\bibitem{Lac2015}
Michael~T. Lacey, \emph{An elementary proof of the {$A_2$} bound}, Israel J.
  Math. \textbf{217} (2017), no.~1, 181--195. \MR{3625108}

\bibitem{LMena}
Michael~T.\ Lacey and Dar{\`\i}o Mena~Arias, \emph{The sparse ${T}(1)$
  theorem}, Houston J.\ Math. \textbf{43} (2017), no.~1, 111--127.

\bibitem{LS}
Michael~T.\ Lacey and Scott Spencer, \emph{Sparse bounds for oscillatory and
  random singular integrals}, New York J. Math. \textbf{23} (2017), 119--131.

\bibitem{LerNaz2015}
Andrei Lerner and Fedor Nazarov, \emph{Intuitive dyadic calculus: the basics},
  preprint arXiv:1508.05639 (2015).

\bibitem{Ler2013}
Andrei~K. Lerner, \emph{A simple proof of the {$A_2$} conjecture}, Int. Math.
  Res. Not. IMRN (2013), no.~14, 3159--3170. \MR{3085756}

\bibitem{Ler2015}
\bysame, \emph{On pointwise estimates involving sparse operators}, New York J.
  Math. \textbf{22} (2016), 341--349. \MR{3484688}

\bibitem{Li1}
Kangwei Li, \emph{Sparse domination theorem for multilinear singular integral
  operators with {${L}^r$}-{H}{\"o}rmander condition}, preprint
  arXiv:1606.03340.

\bibitem{LPR}
Teresa Luque, Carlos P{\'e}rez, and Ezequiel Rela, \emph{Optimal exponents in
  weighted estimates without examples}, Math. Res. Lett. \textbf{22} (2015),
  no.~1, 183--201. \MR{3342184}

\bibitem{KM}
Kabe Moen, \emph{Sharp weighted bounds without testing or extrapolation}, Arch.
  Math. (Basel) \textbf{99} (2012), no.~5, 457--466. \MR{3000426}

\bibitem{PPR}
Carlos P{\'e}rez, Israel Rivera-Rios, and Luz Roncal, \emph{{$A_1$} theory of
  weights for rough homogeneous singular integrals and commutators}, preprint
  arXiv:1607.06432.

\bibitem{Seeger}
Andreas Seeger, \emph{Singular integral operators with rough convolution
  kernels}, J. Amer. Math. Soc. \textbf{9} (1996), no.~1, 95--105. \MR{1317232}

\bibitem{ShiSun}
Xian~Liang Shi and Qi~Yu Sun, \emph{Weighted norm inequalities for
  {B}ochner-{R}iesz operators and singular integral operators}, Proc. Amer.
  Math. Soc. \textbf{116} (1992), no.~3, 665--673. \MR{1136237}

\bibitem{Stein}
Elias~M. Stein, \emph{Harmonic analysis: real-variable methods, orthogonality,
  and oscillatory integrals}, Princeton Mathematical Series, vol.~43, Princeton
  University Press, Princeton, NJ, 1993, With the assistance of Timothy S.
  Murphy, Monographs in Harmonic Analysis, III. \MR{1232192}

\bibitem{V}
Ana~M. Vargas, \emph{Weighted weak type {$(1,1)$} bounds for rough operators},
  J. London Math. Soc. (2) \textbf{54} (1996), no.~2, 297--310. \MR{1405057}

\bibitem{Wat1990}
David~K. Watson, \emph{Weighted estimates for singular integrals via {F}ourier
  transform estimates}, Duke Math. J. \textbf{60} (1990), no.~2, 389--399.
  \MR{1047758}

\end{thebibliography}
\bibliographystyle{amsplain}
\end{document}